\documentclass[11pt,reqno]{amsart}
\usepackage{amssymb}
\usepackage{mathptmx}
\usepackage{mathtools}
\usepackage{mathrsfs}
\usepackage[export]{adjustbox}
\usepackage[all]{xy}
\usepackage{stmaryrd}
\usepackage{fancyhdr}
\usepackage{hyperref}
\usepackage{tikz-cd}
\usepackage{cite}
\usepackage{amsmath,amsfonts}
\usepackage{graphicx}
\usepackage{wrapfig}
\usepackage{array}
\usepackage{amsthm}
\usepackage[left=1.2in, right=1.2in, top=1in, bottom=1in]{geometry}
\usepackage{etoolbox}
\patchcmd{\section}{\scshape}{\bfseries}{}{}
\makeatletter
\renewcommand{\@secnumfont}{\bfseries}
\makeatother
\xyoption{all}
\theoremstyle{plain}

\newcommand{\mc}{\mathcal}

\newcommand{\Hom}{\operatorname{Hom}}

\newcommand{\rk}{\operatorname{rk}}  
\newcommand{\ns}{\operatorname{null}}

\newcommand{\angles}[1]{\left\langle #1 \right\rangle}

\newcommand{\Mat}{\textbf{Mat}_\bullet}

\newcommand{\FMat}{F\text{-}\textbf{Mat}_\bullet}

\newcommand{\TRS}{\textbf{TRS}_\bullet^\Sigma}
\newcommand{\MTRS}{\textbf{MTRS}_\bullet^\Sigma}

\newcommand{\EE}{\mathfrak{E}}          
\newcommand{\MM}{\mathfrak{M}}          

\input xy
\xyoption{all}

\theoremstyle{definition}
\newtheorem{mydef}{\textbf{Definition}}[section]
\newtheorem{myeg}[mydef]{\textbf{Example}}

\newtheorem{rmk}[mydef]{\textbf{Remark}}

\theoremstyle{plain}
\newtheorem{mythm}[mydef]{\textbf{Theorem}}

\newtheorem*{nothma}{\textbf{Theorem A}}
\newtheorem*{nothmb}{\textbf{Theorem B}}
\newtheorem*{nothmc}{\textbf{Theorem C}}
\newtheorem*{nothmd}{\textbf{Theorem D}}

\newtheorem{mytheorem}[mydef]{\textbf{Theorem}}
\newtheorem{lem}[mydef]{\textbf{Lemma}}
\newtheorem{pro}[mydef]{\textbf{Proposition}}

\newtheorem{cor}[mydef]{\textbf{Corollary}}

\tikzset{main node/.style={circle,fill=black,draw,minimum size=0.3cm,inner sep=0pt},
}

\begin{document}

	\title{Proto-exact categories of matroids over idylls and tropical toric reflexive sheaves}

	\author{Jaiung Jun}
	\address{Department of Mathematics, State University of New York at New Paltz, NY, USA}
	\email{junj@newpaltz.edu}
	
	\author{Alex Sistko}
	\address{Department of Mathematics, State University of New York at New Paltz, NY, USA}
	\email{sistkoa@newpaltz.edu}
	
	\author{Cameron Wright}
	\address{Department of Mathematics, University of Washington, WA, USA}
	\email{wrightc8@uw.edu}

	\subjclass[2010]{}
	
	\dedicatory{}
	
\makeatletter
\@namedef{subjclassname@2020}{%
	\textup{2020} Mathematics Subject Classification}
\makeatother

\subjclass[2020]{}
	
	\maketitle
	
	
\begin{abstract}
We study the category $F$-$\textbf{Mat}_\bullet$ of matroids over an idyll $F$. We show that $F\text{-}\textbf{Mat}_\bullet$ is a proto-exact category, a non-additive generalization of an exact category by Dyckerhoff and Kapranov. We further show that $F\text{-}\textbf{Mat}_\bullet$ is proto-abelian in the sense of Andr\'e. As an application, we establish that the category $\textbf{TRS}_\bullet^\Sigma$ of tropical toric reflexive sheaves associated to a fan $\Sigma$, introduced by Khan and Maclagan, is also proto-exact and proto-abelian. We then investigate the stability of modular tropical toric reflexive sheaves within the framework of proto-abelian categories and reformulate Harder–Narasimhan filtrations in this setting.
\end{abstract} 

	
\section{Introduction}

In recent years, combinatorial structures such as graphs, matroids, and simplicial complexes have played an increasingly central role in various areas of mathematics, including algebraic geometry, representation theory, and topology. These structures are not only of intrinsic combinatorial interest but also serve as models that encode essential features of geometric and algebraic objects. Meanwhile, methods from geometry and category theory have led to surprising breakthroughs in the study of combinatorial invariants. The current paper aims to advance this growing interface by developing a categorical framework for understanding combinatorial phenomena, grounded in the theory of proto-exact categories.

A proto-exact category is a pointed category $\mathcal{C}$ with two distinguished classes of morphisms $\mathfrak{M}$ (admissible monomorphisms) and $\mathfrak{E}$ (admissible epimorphisms) satisfying some conditions. Proto-exact categories were first introduced by Dyckerhoff and Kapranov \cite{dyckerhoff2019higher} in the context of Hall algebras and 2-Segal spaces, and they have become powerful tools in algebra and representation theory.

One notable feature of proto-exact structures is that they allow the construction of Hall algebras in non-additive settings, such as categories of matroids or finite pointed sets. As in the additive case, a proto-exact category admits a Waldhausen $S_\bullet$-construction, which yields a $2$-Segal simplicial groupoid. This structure provides a foundational framework for developing algebraic $K$-theory and for studying various Hall algebras associated with the category. The power of proto-exact categories lies in their capacity to formalize extension-like behavior in discrete contexts, making them particularly well-suited for analyzing combinatorial structures through a categorical perspective. For further details, see \cite{jun2020toric}, \cite{jun2020quiver}, \cite{jun2023proto}, \cite{jun2024coefficient}, \cite{eberhardt2022group}, \cite{mozgovoy2025proto}, \cite{sistko2025semisimple}.

With this categorical language to study combinatorial objects, our main focus is on \emph{matroids over idylls}. A matroid is a combinatorial structure generalizing linear independence from vector spaces to arbitrary finite sets. It consists of a finite set and a collection of subsets (called independent sets) satisfying axioms mimicking properties of linearly independent vectors, allowing one to study independence in a wide range of mathematical and applied contexts. In their celebrated work \cite{baker2017matroids}, Baker and Bowler introduced the notion of \emph{matroids over idylls}, inspired by earlier work of Dress and Wenzel \cite{dress1991grassmann}.\footnote{In \cite{baker2017matroids} and subsequent works, the notion is developed in a broader context. However, in the current paper, we focus only on the case of idylls.}

An \emph{idyll} is a multiplicative abelian group $F$ equipped with a proper ideal $N_F \subseteq \mathbb{N}[F]$\footnote{Here, $N_F$ is an ideal in the group semiring $\mathbb{N}[F]$.} satisfying two conditions: (1) $0 \in N_F$, and (2) for every $a \in F$, there exists a unique $b \in F$ such that $a + b \in N_F$. Intuitively, $N_F$ captures a notion of linear dependence over $\mathbb{N}$. A matroid over an idyll $F$ is then defined via a generalization of Grassmann–Plücker functions taking values in $F$. Consequently, an idyll $F$ is said to be \emph{perfect} if vectors are orthogonal to covectors for all matroids over $F$. For details, see \cite[Section 1.3]{jarra2024quiver}.

Another central theme of this paper concerns \emph{tropical toric reflexive sheaves} and \emph{tropical toric vector bundles}. In tropical geometry, an algebraic variety $X$ is studied through its combinatorial shadow $\textrm{Trop}(X)$, known as the tropicalization of $X$. This process turns $X$ into a polyhedral complex, which provides a simpler combinatorial model while still reflecting much of the structure of the original variety. Matroids naturally appear in tropical geometry as a combinatorial counterpart of linear spaces. 

The theory of line bundles in tropical geometry, together with its applications to the study of algebraic curves, has been one of the driving forces behind the field. In contrast, the theory of higher-rank vector bundles in tropical geometry remains far less developed. 

In \cite{khan2024tropical}, Khan and Maclagan introduced a notion of tropical toric vector bundles by reformulating Klyachko’s classification of toric vector bundles \cite{klyachko1990equivariant}, wherein vector spaces and their filtrations are replaced by valuated matroids and filtrations by flats of the underlying matroids.

As with many other combinatorial objects, the categorical aspects of matroids remain relatively underexplored, largely because both weak and strong morphisms between matroids tend to behave poorly (see \cite{heunen2017category}). Nevertheless, interesting categorical results can still be obtained. For example, \cite{eppolito2022infinite} shows that finite matroids, when viewed in the category of infinite matroids as defined in \cite{bruhn2013axioms}, are precisely the finitely presentable objects, and it characterizes a class of infinite matroids arising as colimits of finite matroids.

In \cite{eppolito2018proto}, with Eppolito and Szczesny, the first author took a step forward by examining the category $\mathbf{Mat}_{\bullet}$ of pointed matroids\footnote{For technical reasons, we work with pointed matroids rather than matroids.} and strong maps, and established several categorical properties of $\mathbf{Mat}_{\bullet}$, which make it possible to define its Hall algebra, following the framework of Dyckerhoff and Kapranov \cite{dyckerhoff2019higher}. In particular, in \cite{eppolito2018proto},  it is proved that $\mathbf{Mat}_{\bullet}$ has the structure of a finitary proto-exact category. Furthermore, the Hall algebra of $\mathbf{Mat}_{\bullet}$ is a Hopf algebra dual to a matroid-minor Hopf algebra. The same authors further generalize the notion of matroid-minor Hopf algebras to the case of matroids over ``hyperfields'' in \cite{eppolito2020hopf}.\footnote{A hyperfield can be considered as a special class of idylls.}

In this paper, we continue the categorical exploration of matroids, pursuing two principal aims. First, we extend the aforementioned result to the setting of matroids over idylls by employing the notion of submonomial matrices, introduced by Jarra, Lorscheid, and Vital in \cite{jarra2024quiver}.\footnote{See also \cite{iezzi2023tropical} for the case of valuated matroids.} As a direct application, we show that several categories arising from tropical reflexive sheaves are proto-exact. 

The second aim of this paper concerns stability phenomena in a non-additive setting. The notion of stability, first introduced by Mumford in the context of geometric invariant theory, has since become a fundamental tool in the study of moduli spaces. Stable and semistable objects have subsequently been defined in a wide range of mathematical contexts. In the theory of Hall algebras, for instance, stability plays a central role through Harder–Narasimhan filtrations; see \cite{reineke2003harder} and \cite{schiffmann2006lectures}. 

Our initial motivation for the paper was to introduced a notion of stability for matroids over an idyll, such that the Harder-Narasimhan filtration in \cite{khan2024tropical} arises naturally as a special case within the general framework of proto-exact categories. A major difficulty is that in \cite{khan2024tropical}, the uniqueness and existence of Harder-Narasimhan filtrations hold only under a specific combinatorial condition (modularity of flats). Moreover, we have found examples where, even when there exists an epi-monic morphism $f:\mathcal{M} \to \mathcal{N}$ between tropical toric reflexive sheaves, the inequality $\mu(\mathcal{M}) \leq \mu(\mathcal{N})$ does not hold, where the slope function $\mu$ is defined in \cite{khan2024tropical} via a tropical analogue of the first Chern class in the classical setting. This complicates the interpretation of stability in categorical terms. In this paper, we overcome this issue by combining the categorical approach to slope filtrations by Andr\'e \cite{andre2009slope}\footnote{In \cite{andre2009slope}, Andr\'e defined a notion of proto-abelian categories (Definition \ref{definition: Andre's proto-abelian}) as a tool to combine various existing notions of slope filtrations. We emphasize here that Andr\'e's definition of proto-abelian categories and Dyckerhoff-Kapranov formalism of proto-abelian categories share certain features but one does not imply the other in general.} and Li \cite{li2023categorification}. We hope that this perspective will further clarify the structure of the moduli space of matroids introduced by Baker and Lorscheid \cite{baker2021moduli}.

\subsection{Summary of results}

In Section \ref{section: prelim}, we start by recalling basic definitions for matroids over idylls, tropical toric vector bundles, proto-exact categories in the sense of Dyckerhoff and Kapranov \cite{dyckerhoff2019higher} and their Hall algebras and $K$-theory, and proto-abelian categories in the sense of Andr\'e \cite{andre2009slope} and Harder-Narasimhan theory developed by Li \cite{li2023categorification}.

In \cite{eberhardt2022group}, the notions of a proto-exact category with exact direct sum and a combinatorial proto-exact category were introduced. A proto-exact category with \emph{exact direct sum} fulfills a minimal set-up in which Quillen's Group Completion Theorem can be formulated. A combinatorial proto-exact category means that any subobject $U$ of $X\oplus Y$ splits as the direct sum of subobjects $U\cap X \subseteq X$ and $U \cap Y \subseteq Y$. This is typically satisfied in a combinatorial setting, for instance, vector bundles over a monoid scheme.\footnote{In fact, this condition was first considered by the first author and Szczesny in \cite{jun2020toric} to ensure a bialgebra structure on the Hall algebra $H_\mathcal{C}$ of a proto-exact category $\mathcal{C}$. See \cite[Section 2.2]{jun2020toric}.} 

For an idyll $F$, a morphism $f:M \to N$ between $F$-matroids are defined by a submonomial matrix over $F$ which sends vectors of $M$ to vectors of $N$ (Definition \ref{definition: morphism F-matroids}). In Section \ref{sec:category_FMat}, we study various properties of the category $F$-$\Mat$ of matroids over an idyll, and by appealing to these results, we prove the following. 

\begin{nothma}[Theorems \ref{thm:F mat proto ex} and \ref{theorem: further structure}]
For each perfect idyll $F$, the category $F$-$\Mat$ with $\mathfrak{M}$ matroid restrictions and $\mathfrak{E}$ matroid contractions is a combinatorial proto-exact category with duality and with exact direct sum.
\end{nothma}

In Section \ref{sec: simple F matroids}, we further explore categorical properties of the full subcategory $F$-$\mathbf{SMat}_\bullet$ of $F$-$\Mat$ consisting of $F$-matroids whose underlying matroids are simple. Among other things, we prove the following. 

\begin{nothmb}
For a perfect idyll $F$, the category $F$-$\mathbf{SMat}_\bullet$ is proto-abelian in the sense of Andr\'e (Definition \ref{definition: Andre's proto-abelian}).
\end{nothmb}

In Section \ref{section: Hall algebra and K-theory}, we discuss two immediate consequences of Theorem A. From Theorem A, we can define the Hall algebra and algebraic $K$-theory of $F$-$\Mat$. We prove that the Hall algebra $H_{F\text{-}\Mat}$ of $F$-$\Mat$ for a finite hyperfield $F$ (viewed as a idyll) is isomorphic to the graded Hopf dual of the $F$-matroid-minor Hopf algebra defined in \cite{eppolito2020hopf}, generalizing the result in \cite{eppolito2018proto} that $H_{\Mat}$ is isomorphic to the graded Hopf dual of the matroid-minor Hopf algebra to the case of finite hyperfields. We further prove that as in the case of $\Mat$, $K_0(F\text{-}\Mat)$ is isomorphic to $\mathbb{Z} \oplus \mathbb{Z}$. 

In Section \ref{sec: trs proto ex}, we turn our attention to tropical toric reflexive sheaves (Definition \ref{definition: tropical toric reflexive sheaf}), which are defined as a valuated matroid with some extra data, and tropical toric vector bundles  as a special class of them (Definition \ref{definition: tropical toric vector bundle}). Note that valuated matroids serve as analogues of linear spaces in tropical geometry and can be identified with matroids over the idyll $\mathbb{T}$, known as the \emph{tropical hyperfield}. Therefore, it is natural to find potential applications of our categorical investigation of $F$-$\mathbf{Mat}_{\bullet}$ in tropical toric reflexive sheaves and tropical toric vector bundles. 

In \cite{khan2024tropical}, Khan and Maclagan did not consider categorical aspects of tropical toric reflexive sheaves. They only defined a notion of isomorphism (\cite[Definition 6.7]{khan2024tropical}) motivated by the result of Klyachko \cite[Corollary 1.2.4]{klyachko1990equivariant} that for two indecomposable toric vector bundles $\mathcal{E}$ and $\mathcal{F}$ on a complete toric variety, if $\mathcal{E}$ and $\mathcal{F}$ are isomorphic as ordinary bundles, then $\mathcal{E}$ is equivariantly isomorphic to $\mathcal{F}\otimes \text{div}(\chi^m)$ for some character $\chi$. So, we start by defining a notion of morphisms for tropical toric reflexive sheaves (Definition \ref{definition: morphism of trs}) to define the category $\TRS$ of tropical toric reflexive sheaves associated to a fan $\Sigma$. We then confirm that in $\TRS$ isomorphisms are precisely isomorphisms as defined in \cite{khan2024tropical} (Proposition \ref{pro: trs isos}).

To define proto-exact structure for $\TRS$, we define $\mathfrak{M}_t$ (admissible monos) in terms of restriction on flats and $\mathfrak{E}_t$ (admissible epis) in terms of contractions by flats (Definition \ref{definition: admissible}). We also consider the full subcategory $\MTRS$ of $\TRS$ consisting of tropical toric reflexive sheaves whose underlying matroids are modular.

\begin{nothmc}[Theorem \ref{theorem: proto toric} and Corollary \ref{corollary: MTRS proto-abelian}] 
The following hold.
\begin{enumerate}
	\item 
$\TRS$ is a proto-exact category with $\mathfrak{M}_t$ and $\mathfrak{E}_t$ as above. 
\item 
$\TRS$ is proto-abelian (in the sense of Andr\'e). 
\item 
The proto-exact and proto-abelian structures on $\TRS$ descends to $\MTRS$, yielding that it is also proto-exact and proto-abelian (in the sense of Andr\'e). 
\end{enumerate} 	
\end{nothmc}

Finally, we move on to the slope stability of $\MTRS$. In \cite{li2023categorification}, Li studied categorification of Harder-Narasimhan filtrations building on the work of Andr\'e on the slope filtrations on proto-abelian categories (Theorem \ref{thm:li HN filtr}), which we call a \emph{categorical slope filtration} in this paper. We then recast the Harder-Narasimhan filtration in \cite{khan2024tropical} for modular tropical reflexive sheaves in this setting. To be precise, we prove the following. 

\begin{nothmd}[Corollary \ref{cor: HN filtr coincide}]
	The Harder-Narasimhan filtrations of Khan-Maclagan associated to modular tropical reflexive sheaves are given by the categorical slope filtration on $\MTRS$ given by Li. In other words, the Harder-Narasimhan filtration of \cite{khan2024tropical} for modular tropical reflexive sheaves is a categorical slope filtration in the sense of \cite{li2023categorification}.
\end{nothmd}

\bigskip

\textbf{Acknowledgment} J.J. acknowledges AMS-Simons Research Enhancement Grant for Primarily Undergraduate Institution (PUI) Faculty during the writing of this paper, and was a member of the Institute for Advanced Study supported by the Bell System Fellowship Fund during most of the period in which this work was carried out.

\section{Preliminaries} \label{section: prelim}

\subsection{Bands and Idylls}

Bands and idylls are generalizations of rings and fields, respectively, which are of fundamental interest to us. Following \cite{baker2017matroids} and \cite{jarra2024quiver}, we introduce matroids over idylls and their pointed versions, as well as a handful of concepts promoted from classical matroid theory to matroids over idylls. The theory of matroids we encounter here follows the works \cite{baker2017matroids} and \cite{jarra2024quiver}. For more details on bands, idylls, and relations to other partial algebraic structures, see \cite{baker2025new}. 

A \emph{pointed monoid} is a (multiplicative) commutative monoid $B = B'\sqcup \{0\}$ such that $B'$ is a submonoid and $0$ is an absorbing element, i.e., $0\cdot b = 0$ for all $b\in B$. To each pointed monoid $B$ is associated an \emph{ambient semiring} $B^+ = \mathbb{N}[B]/\langle 0\rangle$, where $\langle0 \rangle$ is the semiring ideal generated by $0$. A \emph{band} is a pointed monoid equipped with a \emph{null set} $N_B$, which is a proper ideal of $B^+$ satisfying 
\begin{enumerate}
	
	\item[(B0)] $0\in N_B$
	
	\item[(B1)] For every $a\in B$ there exists a unique element $b\in B$ such that $a+b\in N_B$. We denote this element by $-a$. In particular, note that each band contains an element $-1$ satisfying $1 + (-1) \in N_B$, where $1$ is the multiplicative unit of $B$.
	
\end{enumerate}

We think of the null set of a band $B$ as capturing a notion of linear dependence over $\mathbb{N}$. An \emph{idyll} is a band $F$ such that $F^\times = F\setminus \{0\}$. A morphism of bands $\Phi:B\to C$ is a morphism of monoids with $\Phi(0_B)=0_C$ such that $\sum_i \Phi(b_i)\in N_C$ for all $\sum_i b_i\in N_B$. With this notion of morphisms, bands form a category and idylls form a subcategory thereof. 

\begin{myeg}
	\begin{enumerate}
		\item 
Any commutative ring $R$ is a band with $N_R = \{\sum_i a_i : \sum_i a_i=0\}$.
		\item 
The \emph{Krasner hyperfield} is the idyll $\mathbb{K}$ with underlying set $\{0,1\}$ whose multiplicative structure is inherited from $\mathbb{Z}$ and whose nullset is $\mathbb{N}\setminus\{1\}$. This is the final object in the category of idylls.
		\item The \emph{regular partial field} is the idyll $\mathbb{F}_1^{\pm}$ with underlying set $\{-1,0,1\}$ whose multiplicative structure is inherited from $\mathbb{Z}$ and whose null set is $N_{\mathbb{F}_1^{\pm}} = \langle 1 + (- 1) \rangle$.
		\item The \emph{tropical hyperfield} is the idyll $\mathbb{T}$ with underlying set $\mathbb{R}\sqcup \{\infty\}$ and multiplicative structure given by the additive group $\mathbb{R}$ and with additive structure defined by $a+b := \max\{a,b\}$. The null set is given by $N_\mathbb{T} = \{\sum_i a_i : \text{the maximum is attained at least twice}\}$.
		\item The \emph{hyperfield of signs} is the idyll $\mathbb{S}$ with underlying set $\{-1,0,1\}$ whose multiplicative structure is again inherited from $\mathbb{Z}$ but whose null set is defined by
		$N_\mathbb{S} = \{m\cdot 1 + n\cdot (-1) : m=n=0\text{ or }m\neq 0 \neq n\}$. 
	\end{enumerate}
\end{myeg}

We will be particularly interested in the case where $B=F$ is a \emph{perfect} idyll. In short, an idyll $F$ is perfect precisely when the vectors of any $F$-matroid are orthogonal to the covectors of that $F$-matroid. The idylls $\mathbb{F}_1^{\pm}$, $\mathbb{K}$, $\mathbb{T}$, and $\mathbb{S}$ are perfect, as are all fields $\mathbf{k}$.

For ordered tuples of elements of a set, we use the following notation. Let $E$ be a set and take some $n\in\mathbb{N}$. Then for a vector $\mathbf{x} \in E^n$, write $\mathbf{x}_k$ for the $k$th entry of $\mathbf{x}$ and write $\mathbf{x}_{\hat{k}}$ for the vector of length $n-1$ obtained from $\mathbf{x}$ by deleting the $k$th entry. Write $|\mathbf{x}|$ for the \emph{support} of $\mathbf{x}$, which is the set of nonzero elements of $F$ appearing as entries of $\mathbf{x}$.

\subsection{Matroids over Idylls}\label{subsec:F matroids}

We define matroids over idylls as equivalence classes of Grassmann-Pl\"{u}cker functions, following notation and definitions of \cite{baker2017matroids} and \cite{jarra2024quiver}. For the duration of this section $F$ will always denote an idyll. A \emph{Grassmann-Pl\"{u}cker function} of rank $r$ on a finite ground set $E$ over $F$ is a function $\mu:E^r\to F$ satisfying the following properties:
\begin{enumerate}
	\item[(GP0)] 
If $r>0$ then $\mu$ is not identically zero.
	\item[(GP1)] 
	$\mu$ is alternating in the sense that $\mu(\mathbf{x}) = \text{sign}(\sigma)(\mathbf{x}^\sigma)$ for each permutation $\sigma\in S_r$ and $\mu(\mathbf{x})=0$ whenever there are $i\neq j$ with $x_i = x_j$.	
	\item[(GP2)]
	 For every $\mathbf{x}=(x_0,\ldots,x_r)\in E^{r+1}$ and every $\mathbf{y} = (y_2,\ldots,y_r)\in E^{r-1}$, we have that
	\[\sum_{k=0}^r (-1)^k\mu(\mathbf{x}_{\hat{k}})\cdot\mu(x_k,\mathbf{y}) \in N_F.\]
\end{enumerate}

By convention, if $r=0$ and $E$ is a pointed set with basepoint $*$, we set $E^0:= \{*\}$. We say that two Grassmann-Pl\"{u}cker functions $\mu$ and $\mu'$ are \emph{equivalent} if there is some $\lambda\in F^\times$ such that $\mu' = \lambda\cdot \mu$. 

\begin{mydef}
An $F$-\emph{matroid} over $F$ of rank $r$ is an equivalence class $M = [\mu]$ of a Grassmann-Pl\"{u}cker function $\mu$ of rank $r$.\footnote{Strictly speaking, these are \emph{strong} $F$-matroids as in \cite{baker2017matroids}. But since we only consider a perfect idyll $F$, two notions of strong and weak $F$-matroids agree in our setting. }
\end{mydef} 

In the case that $F=\mathbb{K}$ is the Krasner hyperfield, an $F$-matroid of rank $r$ on $E$ is a matroid of rank $r$ on the set $E$ in the classical sense. Likewise, a $\mathbb{T}$-matroid is a \emph{valuated matroid} and an $\mathbb{S}$-matroid is an \emph{oriented matroid}. For a field $\mathbf{k}$, a $\mathbf{k}$-matroid of rank $r$ on $E$ is an $r$-dimensional linear subspace of $\mathbf{k}^E$. See \cite{baker2017matroids} for more on these characterizations.

Let $M=[\mu]$ be an $F$-matroid of rank $r$ on the set $E_M$ and fix a total order on $E_M$. Let $n=\#|E_M|$. The \emph{dual} of $M$ is the $F$-matroid $M^*=[\mu^*]$, where $\mu^*:E_M^{n-r}\to F$ is the Grassmann-Pl\"{u}cker function defined by 
\[\mu^*(\mathbf{x}) = \begin{cases} 0 & \exists i\neq j\text{ with }x_i=x_j, \\ \text{sign}(\mathbf{x},\mathbf{x}')\cdot\mu(\mathbf{x}') & \text{else} ,\end{cases}\]
where $\mathbf{x}'\in E_M^r$ is such that the support of $\mathbf{x}$ and $\mathbf{x}'$ form a partition of $E_M$ and the function $\text{sign}(\cdot)$ assigns the sign of the permutation of $E_M$ defined by the string $(\mathbf{x},\mathbf{x}')$. The dual $M^*$ is independent of the choice of total order on $E_M$. If $M$ is a pointed matroid then we define $M^*$ to be the matroid obtained by deleting $0_M$, taking the dual, and appending a loop $0_{M^*}$.

\begin{rmk}
	Some idylls $F$ come equipped with a natural involution which interact directly with the structure of matroids over $F$. For example, the hyperfield of signs $\mathbb{S}$ is equipped with the involution which exchanges + and -. For $F$-matroids over idylls with involution $\overline{(-)}$, the dual matroid is instead defined by 
	\[\mu^*(\mathbf{x}) = \begin{cases} 0 & \exists i\neq j\text{ with }x_i=x_j, \\ \text{sign}(\mathbf{x},\mathbf{x}')\cdot\overline{\mu(\mathbf{x}')} & \text{else}. \end{cases}\]
	See \cite{jarra2024quiver} Subsection 2.7 for more details.
\end{rmk}

\subsubsection{Circuits and vectors of $F$-matroids}

Following \cite{baker2017matroids} and \cite{anderson2019vectors}, to each $F$-matroid $M$ of rank $r$ is associated collections $\mathcal{C}_M, \mathcal{V}_M\subseteq F^r$ of \emph{circuits} and \emph{vectors}, respectively. In the case that $F=\mathbb{K}$ (resp.~$F=\mathbb{S}$), these are the circuit and vector sets of the associated (resp.~oriented) matroid. We will use these notions only in the next subsection in order to define morphisms of matroids, and mention them primarily to emphasize that this theory is a direct generalization of classical (oriented) matroid theory. As such, we refer the reader to \cite{baker2017matroids}, \cite{anderson2019vectors}, and \cite{jarra2024quiver} for more details.

For set $E$, we let $F^E$ be the set of vectors indexed by $E$, i.e., $F^E=\{X=(X_e)_{e\in E} \mid X_e \in F\}$. For $X \in F^E$, the \emph{support} of $X$ is the set $\underline{X}:=\{e \in E \mid X_e \neq 0_F\}$ and for a subset $A \subseteq F^E$, the support of $A$ is $\underline{A}:=\{\underline{X} \mid X \in A\}$. For $X,Y \in F^E$, we say that $X$ and $Y$ are \emph{orthogonal}, denoted by $X \perp Y$, if
\[
X\cdot Y:=\sum_{e \in E} X_e \cdot Y_e \in N_F,
\]
where $X_e\cdot Y_e$ is the multiplication in $F$ and $N_F$ is the null set of $F$. For $A \subseteq F^E$, we let $A^\perp$ be the set of vectors $X \in F^E$ which are orthogonal to all vectors in $A$. For $A,B \subseteq F^E$, we write $A \perp B$ if $X \perp Y$ for any $X \in A$ and $Y \in B$. 

Suppose that $M=[\mu]$ is an $F$-matroid of rank $r$ and define $\Xi_M\subseteq E_M^{r+1}$ to be the set of $\mathbf{y}=(y_0,\ldots,y_r)\in E_M^{r+1}$ for which $\#|\mathbf{y}| = r+1$ and such that there is some $i\in \{0,1,\dots, r\}$ for which $\mu(\mathbf{y}_{\hat{i}})\neq 0$. For any $\mathbf{y}\in \Xi_M$ define the \emph{fundamental circuit} of $\mathbf{y}$ with respect to $\mu$ to be the vector $X_\mathbf{y}\in F^{E_M}$ given by 
	\[X_\mathbf{y}(x) = \begin{cases} (-1)^k\cdot\mu(\mathbf{y}_{\hat{k}}) & x=y_k, \\ 0 & \text{else.} \end{cases}\]
Then the \emph{set of circuits} of $M$ is $\mathcal{C}_M = \{a\cdot X_\mathbf{y}\;|\;a\in F^\times,\;\mathbf{y}\in\Xi_M\}$ and the \emph{set of cocircuits} of $M$ is defined to be $\mathcal{C}^*_M := \mathcal{C}_{M^*}$. Correspondingly, define the \emph{set of vectors} of $M$ to be $\mathcal{V}_M = (\mathcal{C}^*_M)^\perp$ and the \emph{set of covectors} to be  $\mathcal{V}^*_M=\mathcal{C}_M^\perp$. An idyll $F$ is said to be \emph{perfect} when $\mathcal{V}_M\perp\mathcal{V}^*_M$ for all $F$-matroids $M$.

Given an $F$-matroid $M=[\mu]$, a morphism $\Phi:F\to G$ of idylls yields a $G$-matroid $\Phi_*M=[\Phi\circ\mu]$ on the same ground set, referred to as the push-forward of $M$ along $\Phi$. The push-forward of any $F$-matroid $M$ along the unique map $F\to \mathbb{K}$ is referred to as the \emph{underlying matroid} of $M$ and is denoted by $\underline{M}$. Note that $\underline{M}$ is a classical matroid. 

\subsubsection{Direct sum and pointed $F$-matroids}\label{subsubsec:sums and points}
For any idyll $F$ there are precisely two $F$-matroids defined on a one-element set. The first is $U_1^0$, whose associated Grassmann-Pl\"{u}cker function takes only the value $0_F$. The second is $U_1^1$, whose Grassmann-Pl\"{u}cker function takes any value in $F^\times$. Since $F$ is an idyll, the group $F^\times$ consists of all nonzero elements of $F$ and so the Grassmann-Pl\"{u}cker function of $U_1^1$ is unique up to nonzero scaling, hence $U_1^1$ is unique. 

Given two $F$-matroids $M=[\mu]$ and $N=[\nu]$ of ranks $r_M$ and $r_N$ on ground sets $E_M$ and $E_N$, the \emph{direct sum} of $M$ and $N$ is defined to be $F$-matroid given as the equivalence class $M\oplus N := [\mu\oplus\nu]$, where $\mu\oplus\nu:(E_M\sqcup E_N)^{r_M+r_N}\to F$ is a function defined as follows. For $\mathbf{x}\in(E_M\sqcup E_N)^{r_M+r_N}$, let $\mathbf{x}_M$ be the sub-tuple consisting of those entries of $\mathbf{x}$ which are in $E_M$, and define $\mathbf{x}_N$ similarly. Then define $\mu\oplus\nu$ by
\[(\mu\oplus\nu)(\mathbf{x}) := \begin{cases} \mu(\mathbf{x}_M)\cdot \nu(\mathbf{x}_N) & \#|\mathbf{x}_M| = r_M\text{ and }\#|\mathbf{x}_N|=r_N \\ 0 & \text{else} \end{cases}.\]

A \emph{pointed matroid} of rank $r$ over $F$ is a matroid $M=[\mu]$ in the usual sense which is defined on a pointed ground set $E_M = \tilde{E}_M\sqcup \{*_M\}$ for which $\mu(\mathbf{x})\in N_F$ whenever $\mathbf{x}\in E_M^r$ is such that $\mathbf{x}_i = *_M$ for some $i$. For the case that $F=\mathbb{K}$, this means that $*_M$ is a loop of the underlying matroid. In other words, a pointed matroid over an idyll $F$ is a matroid with the choice of a distinguished loop in the ground set. We refer to $\tilde{E}_M$ as the set of \emph{nonzero elements} of $E_M$. 

Note that from each non-pointed $F$-matroid $M$ we can obtain a pointed $F$-matroid by taking the direct sum with the matroid $U_1^0$, and in fact every pointed matroid arises this way. 

Given two pointed matroids $M$ and $N$, we define the pointed direct sum $M\oplus N$ to be an $F$-matroid on the ground set $\tilde{E}_M\sqcup \tilde{E}_N\sqcup \{*_{M\oplus N}\}$ and proceed to define the associated Grassmann-Pl\"{u}cker function in complete analogy to the above. Alternatively, one can form the pointed direct sum by taking the usual direct sum of pointed $F$-matroids and subsequently identifying the two distinguished loops of the summands. 

The definition of matroids as equivalence classes of Grassmann-Pl\"{u}cker functions is not standard in many classic matroid texts, and so it may seem unintuitive to some familiar with the classical theory. The following result illustrates a concrete relation to the classical theory of matroids. 

\begin{pro}[\cite{baker2017matroids}]\label{prop:classical_matroid_nonzero_GP}
	Let $M=[\mu]$ be an $F$-matroid of rank $r$ on $E_M$ and let $\underline{M}$ be its underlying matroid. Then for any $r$-tuple $\mathbf{e} =(e_1,\ldots,e_r)\in E_M^r$, we have that $\mu(\mathbf{e})\neq0_F$ if and only if $\{e_1,\ldots,e_r\}$ is a basis of $\underline{M}$. 
\end{pro}
\begin{proof}
	This follows from Theorem 3.17 of \cite{baker2017matroids} and the discussion following.
\end{proof}

\subsubsection{Restriction and contraction of $F$-matroids}\label{subsubsec:restriction contraction}
$F$-matroids also possess notions of restrictions and contractions which generalize those of classical matroid theory. Let $M=[\mu]$ be an $F$-matroid of rank $r$ on $E_M$ and $A\subseteq E_M$ an arbitrary subset. We define the \emph{deletion}, \emph{contraction}, and \emph{restriction} of $M$ with respect to $A$ as follows. For each of the following definitions, suppose that $A\subseteq E_M$ is a set with $\rk_{\underline{M}}(A) = r'$ and $\rk_{\underline{M}}(E_M\setminus A) = r''$.
\begin{enumerate}
	\item[(1)] (Contraction) Let $\mathbf{a}\in A^{r'}$ be a tuple of elements of $A$ for which $|\mathbf{a}|$ is a basis of the classical matroid $\underline{M}|A$. Then define the function $\mu/A:(E_M\setminus A)^{r-r'}\to F$ by
	\[(\mu/A)(\mathbf{x}):= \mu(\mathbf{x},\mathbf{a}).\]
	\item[(2)] (Deletion) Let $\mathbf{b}\in A^{r-r''}$ be a tuple of elements of $A$ for which $|\mathbf{b}|$ is a basis of the classical matroid $\underline{M}/(E_M\setminus A)$. Then define the function $\mu\setminus A:(E_M\setminus A)^{r''}\to F$ by
	\[(\mu\setminus A)(\mathbf{x}):= \mu(\mathbf{x},\mathbf{b}).\] 
	\item[(3)] (Restriction) With all notations as above, define $\mu|A:=\mu\setminus(E_M\setminus A):A^{r'}\to F$. More precisely, let $\mathbf{c}\in (E_M\setminus A)^{r-r'}$ be a tuple for which $|\mathbf{c}|$ is a basis for the classical matroid $\underline{M}/A$. Then $\mu|A$ is given by
	\[(\mu|A)(\mathbf{x}) = \mu(\mathbf{x},\mathbf{c}).\]
\end{enumerate} 
The \emph{contraction} of $M$ by $A$ is the $F$-matroid $M/A:=[\mu/A]$, the \emph{deletion} of $M$ by $A$ is $M\setminus A:=[\mu\setminus A]$, and the \emph{restriction} of $M$ to $A$ is $M|A:=[\mu|A]$. In each case, the resulting function is again a Grassmann-Pl\"{u}cker function and the matroids they define independent of any of the choices of $\mathbf{a}$, $\mathbf{b}$, and $\mathbf{c}$ above. With these notions, by \cite[ Theorem 3.29]{baker2017matroids}, we have the usual identities $(M/A)^*=M^*\setminus A$ and $(M\setminus A)^* = M^*/A$. In the sequel, for any $A\subseteq E_M$ we denote by $A^c$ the complement $E_M\setminus A$. 

\begin{rmk}\label{rmk:loops}
	Some care should be taken when performing restriction and contraction on pointed matroids. When restricting to a subset $S$ not containing the distinguished loop, we tacitly append a loop $*_{M|S}$ to the result. Likewise, when contracting a subset $T$ which contains the loop, we contract the subset $T\setminus *_M$. These adjustments do not alter the constructions or the resulting matroids, but free us from worrying about the loop.
\end{rmk}

As in the theory of classical matroids, the operations of contraction and deletion commute when applied to disjoint subsets of the ground set. 

\begin{pro}\label{prop:restriction-contraction}
	Let $M=[\mu]$ be a pointed $F$-matroid of rank $m$ on a set $E$ and let $S\subseteq T\subseteq E$ be subsets of the ground set. Then $(M|T)/S = (M/S)|T$.
\end{pro}
\begin{proof}
We prove in terms of deletion and contraction, establishing instead that for any disjoint $A,B\subseteq E$, we have $(M\setminus A)/B = (M/B)\setminus A$. It is easily verified that this is equivalent to the statement as written. 
	
	The proof proceeds by direct calculation of the associated Grassmann-Pl\"{u}cker functions. Let $A,B\subseteq E$ be disjoint sets with $\ell = \text{rk}_{\underline{M}}(A)$ and $k = \text{rk}_{\underline{M/A}}(E\setminus B)$. 
	
	We first consider $(M/A)\setminus B$. Let $\mathbf{a}\in A^{\ell}$ be a basis of $\underline{M}|A$ and $\mathbf{b}\in B^{m-\ell-k}$ be a basis of $\underline{(M/A)}/(E\setminus B)$. Then $(M/A)\setminus B$ is given by the function $(\mu/A)\setminus B:(E\setminus (A\cup B))^{k}\to F$ which sends $\mathbf{y}\in(E\setminus(A\cup B))^k$ to
	\[((\mu/A)\setminus B)(\mathbf{y})= (\mu/A)(\mathbf{y},\mathbf{b}) = \mu(\mathbf{y},\mathbf{b},\mathbf{a})\]
	Note that the argument of $\mu$ above has $k+(m-\ell-k)+\ell = m$ entries.
	
	To approach $(M\setminus B)/A$, note first that the disjointness of $A$ and $B$ ensures that $\rk_{\underline{M}}(A) = \rk_{\underline{M\setminus B}}(A)$. Thus $\mathbf{a}$ in the previous paragraph serves as a basis for $(\underline{M\setminus B})|A$. Likewise, $\mathbf{b}$ serves as a basis for $\underline{M}/(E\setminus B)$ since it was assumed to be a basis of $(\underline{M/A})/(E\setminus B)$. We now see that $(M\setminus B)/A$ is given by $(\mu\setminus B)/A:(E\setminus(A\cup B))^k\to F$ which sends $\mathbf{y}\in(E\setminus(A\cup B))^k$ to
	\[((\mu\setminus B)/A)(\mathbf{y}) = (\mu\setminus B)(\mathbf{y},\mathbf{a}) = \mu(\mathbf{y},\mathbf{a},\mathbf{b})\]
	Hence we have that $((\mu/A)\setminus B) = \pm ((\mu\setminus B)/A)$. This implies in particular that the associated Grassmann-Pl\"{u}cker functions are scalar multiples of one another and so the matroids they define are identical. 
\end{proof} 

\subsubsection{Lattices and modularity}\label{subsubsec: modularity}

A fundamental tool in the study of classical matroids is the lattice of flats $\mathcal{L}(M)$ of a matroid $M$ (see \cite{oxley2006matroid} for definitions). Flats are certain subsets of the ground set of a matroid which generalize the linear spans of sets of vectors. Given a subset $A$ of the ground set of $M$, denote by $\angles{A}$ the smallest flat of $M$ containing $A$ as a subset. The lattice of flats of $M$ is a poset under the subset relation, and is a lattice with meet and join given by
	\begin{align*}
		F\wedge G &= F\cap G,\\
		F\vee G &= \angles{F\cup G}.
	\end{align*}

It is well-known that the lattice of flats of any matroid is a geometric lattice. Moreover, each geometric lattice corresponds uniquely to a simple matroid. This observation will be of use to us in Section \ref{sec: simple F matroids} where we introduce and study the category of simple $F$-matroids. For an $F$-matroid $M$, we write $\mathcal{L}(M)$ to denote the lattice of flats of the underlying (classical) matroid. 

In the sequel, we use the notion of \emph{modularity} of flats $F\in\mathcal{L}(M)$ in an essential way. For any classical matroid $M$, say that two flats $F,G\in\mathcal{L}(M)$ form a \emph{modular pair} precisely when we have an equality in the submodular inequality for $F$ and and $G$:
	\[\rk_M(F) + \rk(G) = \rk(F\vee G) + \rk(F\wedge G).\]
Say that $F\in\mathcal{L}(M)$ is \emph{modular} if $(F,G)$ is a modular pair for every $G\in\mathcal{L}(M)$. The maximum element $E_M$ and minimum element $\angles{\varnothing}$ of $\mathcal{L}(M)$ are always modular flats, as are all atoms of $M$. We have the following characterization of modular flats.

\begin{pro}[\cite{khan2024tropical}, Proposition 2.4] \label{prop: modular flats}
	For $F\in\mathcal{L}(M)$, the following are equivalent.
	\begin{enumerate}
		\item[(1)] $F$ is a modular flat
		\item[(2)] For all pairs of flats $G\leq H$ we have $(F \wedge H) \vee G = (F \vee G) \wedge H$.
		\item[(3)] For all $G\in\mathcal{L}(M)$ and all $H \leq F$, we have $(F \wedge G) \vee H = F \wedge (G \vee H)$. 
	\end{enumerate} 
\end{pro} 

Say that the matroid $M$ is \emph{modular} if every flat $F\in\mathcal{L}(M)$ is modular. For an arbitrary $F$-matroid $M$, say that $M$ is modular precisely when the underlying matroid $\underline{M}$ is. Modular matroids are closely related to projective geometry, as witnessed by the following characterization:

\begin{pro} [\cite{oxley2006matroid}, Proposition 6.9.1]
	A (classical) matroid $M$ is modular if and only if the simplification of each of its connected components is either a free matroid or the matroid of some finite projective geometry. 
\end{pro}

This proposition will not be used explicitly, but instead is included primarily to emphasize that the class of modular matroids is much smaller than that of all matroids. 

The class of modular matroids is closed under the operation of restriction and contraction of flats. Indeed, given a modular matroid $M$ and a flat $F\in\mathcal{L}_M$, the restriction $M|F$ and contraction $M/F$ are again modular matroids. This follows from the well-known fact that for any $F\in\mathcal{L}(M)$ we have
	\begin{align*}
		\mathcal{L}(M|F) &\cong [\angles{\varnothing}, F] \\
		\mathcal{L}(M/F) &\cong [F, E_M]
	\end{align*} 
where the right-hand side of the above expressions are intervals in the lattice $\mathcal{L}(M)$.  These isomorphisms are canonical, with $\varphi: [\angles{\varnothing}, F] \to \mathcal{L}(M|F)$ given by the identity and $\psi: [F, E_M] \to \mathcal{L}(M/F)$ given by $G \mapsto G\setminus F$.

\subsubsection{Simple $F$-matroids}\label{subsubsec{simple F-matroids}}

For an $F$-matroid $M$, say that an element $e\in E_M$ is a \emph{loop} when it is contained in no basis of $\underline{M}$. Similarly, $e,e'\in E_M$ are \emph{parallel} precisely when $\{e,e'\}$ is dependent in $\underline{M}$ but neither are loops. In accordance with these notions, say that $M$ is \emph{simple} precisely when $\underline{M}$ is. That is, $M$ is simple if and only if $\underline{M}$ contains no loops or parallel elements. When working with pointed matroids, we do not consider the distinguished loop for the purpose of simplicity. That is, a pointed matroid $M$ is simple if and only if the matroid obtained by deleting the distinguished point is simple.

To each matroid, and so to each $F$-matroid $M$, we can associate a simple matroid known as the \emph{simplification} of $M$, and denoted by $\text{si}(M)$ or $\text{si}_\bullet(M)$ in the pointed case. The matroid $\text{si}(M)$ (resp. $\text{si}_\bullet(M)$) is obtained from $M$ by deleting all (non-distinguished) loops and all but one element from each class of parallel elements. This is well-defined up to isomorphism. 

It is a standard fact of matroid theory that simple (classical) matroids are uniquely determined by their lattice of flats, and that each geometric lattice is isomorphic to the lattice of flats of a unique simple matroid.

\subsection{Tropical toric vector bundles}\label{subsec: ttvb} 

The parallel works \cite{khan2024tropical} and \cite{kaveh2024tropical} established a notion of tropical vector bundles on tropical toric varieties. The paper \cite{khan2024tropical} in particular also introduced a notion of tropical reflexive sheaves on such varieties and established that a large class of these objects exhibit a Harder-Narasimhan filtration as in the context of vector bundles on curves. In this subsection we review the basics of this theory, phrased in terms of pointed $\mathbb{T}$-matroids.

Let $\Sigma$ be a simplicial fan in a lattice $L \cong \mathbb{Z}^n$ and $\Lambda=\Hom(L,\mathbb{Z})$ the dual lattice. Fix a field $k$ and denote by $X_\Sigma$ the toric $k$-variety with fan $\Sigma$. Let $\text{trop}(X_\Sigma)$ be the tropical toric variety with fan $\Sigma$ as in \cite{kajiwara2008tropical} or \cite{payne2009analytification}. Denote by $\Sigma(1)$ be the set of rays $\rho$ of $\Sigma$ and, in case $\text{trop}(X_\Sigma)$ is the tropicalization of a toric variety $X_\Sigma$, denote the corresponding torus-invariant divisors on $X_\Sigma$ by $D_\rho$. For each ray $\rho$, we let $v_\rho$ be the first lattice point on $\rho$. Fix a polarization $h =(h_\rho)_\rho\in\mathbb{R}^{\Sigma(1)}$. 

\begin{mydef}\cite{khan2024tropical}\label{definition: tropical toric reflexive sheaf}
	A \emph{pointed tropical toric reflexive sheaf} of rank $r$ on the tropical toric variety $\text{trop}(X_\Sigma)$ is a tuple $\mathcal{F}=(N,\{F^\rho_\bullet\}_{\rho\in\Sigma(1)})$, where 
	\begin{enumerate}
		\item 
		$N$ is a simple pointed $\mathbb{T}$-matroid of rank $r$ on the ground set $E_N=\widetilde{E}_N \sqcup \{*_N\}$, and
		\item 
		for each ray $\rho\in\Sigma(1)$, $F^\rho_\bullet =\{F^\rho_j\}_{j \in \mathbb{Z}}$ is a chain of flats of the underlying matroid $\underline{N}$ such that $F^\rho_j \geq F^\rho_{j+1}$, $F^\rho_j=*_N$ for $j \gg 0$, and $F^\rho_j=E_N$ for $j \ll 0$.  
	\end{enumerate}
	Throughout, we write $\{F^\rho_\bullet\}_\rho$ to denote the collection of all flags $F^\rho_\bullet$ as $\rho$ ranges in $\Sigma(1)$. 
\end{mydef}

\begin{mydef}\cite{khan2024tropical}\label{definition: tropical toric vector bundle}
	A \emph{pointed tropical toric vector bundle} is a pointed tropical toric reflexive sheaf $\mathcal{F}=(N,\{F^\rho_\bullet\}_\rho)$ satisfying the following additional condition: for any maximal cone $\sigma \in \Sigma$ there is a multiset $\textbf{u}(\sigma)$ in the dual lattice $\Lambda$, and a basis $B_\sigma=\{\textbf{w}_\textbf{u} \mid \textbf{u} \in \textbf{u}(\sigma)\}$ of $\underline{N}$ such that for every ray $\rho \leq \sigma$ and $j \in \mathbb{Z}$, we have
	\begin{equation}\label{eq: tropical toric vector bundle}
		F^\rho_j=\bigvee_{\mathbf{u} \in B_\sigma, \mathbf{u}\cdot \mathbf{v}_\rho \geq j} \mathbf{w}_\mathbf{u}.
	\end{equation}
	That is, $F^\rho_j$ is the smallest flat containing all the $\mathbf{w}_\mathbf{u}$ with $\mathbf{u}\cdot \mathbf{v}_\rho \geq j$.
\end{mydef}

\begin{rmk}
	As noted in \cite{khan2024tropical}, for smooth toric varieties $X_\Sigma$, \eqref{eq: tropical toric vector bundle} is equivalent to the following: For each cone $\sigma\in\Sigma$, there is a basis $B_\sigma$ of $\underline{M}$ such that whenever $\rho\leq \sigma$, each flat $F^\rho_j$ is a join of elements of $B_\sigma$.
\end{rmk}

In what follows, all tropical toric reflexive sheaves and tropical toric vector bundles are assumed to be pointed, and we will simply say tropical toric reflexive sheaves or tropical toric vector bundle if the context is clear. The \emph{rank} of a tropical toric reflexive sheaf is defined to be $\rk(\mathcal{F}) := \rk(N)$. 

A tropical toric reflexive sheaf of rank 1 is called a \emph{tropical line bundle}. Since a tropical line bundle has an underlying matroid of rank 1, all tropical line bundles are indeed tropical vector bundles. For a tropical line bundle $\mathcal{L} = (M, \{L^\rho_\bullet\}_\rho)$, the only information contained in the chains $L^\rho_\bullet$ is the indices $a_\rho = \min\{j\;:\;L^\rho_j\neq *_M\}$. Thus a tropical line bundle can be specified by a vector $(a_\rho)\in\mathbb{Z}^{\Sigma(1)}$. 

The paper \cite{khan2024tropical} also defines a way to tropicalize a toric vector bundle on $X_\Sigma$ to obtain a tropical toric vector bundle. In particular, the tropicalization of a line bundle $\mathcal{O}(\sum_{\rho} a_\rho\cdot D_\rho)$ is the tropical line bundle corresponding to the vector $(a_\rho)\in\mathbb{Z}^{\Sigma(1)}$ as above.

\begin{mydef}[\cite{khan2024tropical}, Definition 6.9]\label{def: iso of trs}
	Two tropical toric vector bundles $\mathcal{E}$ and $\mathcal{F}$ are \emph{isomorphic} if the associated valuated matroids are isomorphic and if there is some $u\in\Lambda$ for which we have
	\[F^\rho_j = G^\rho_{j+u\cdot v_\rho}\]
	for every $\rho \in\Sigma(1)$ and every $j\in\mathbb{Z}$, where $v_\rho$ is the first lattice point on $\rho$.
\end{mydef}

The Harder-Narasimhan filtration for tropical toric vector bundles is defined using a slope function. The approach taken mirrors the classical setting for toric vector bundles as established in \cite{klyachko1990equivariant}, expressing the degree of a tropical toric reflexive sheaf in terms of the filtration $F^\rho_\bullet$. Precisely, the \emph{degree} of a tropical toric reflexive sheaf $\mathcal{E} = (M,\{F^\rho_\bullet\}_\rho)$ is defined to be 
\begin{equation}
\deg(\mathcal{E}) = \sum_{\rho\in\Sigma(1)} h_\rho\sum_{j\in\mathbb{Z}} j\cdot(\rk_{\underline{M}}(F^\rho_j) - \rk_{\underline{M}}(F^\rho_{j+1})).
\end{equation}
For a tropical toric reflexive sheaf $\mathcal{E}$ as above and a tropical toric line bundle $\mathcal{L}$ corresponding to a vector $(a_\rho)\in \mathbb{Z}^{\Sigma(1)}$, define the \emph{tensor product} of $\mathcal{E}$ and $\mathcal{L}$ to be the tropical toric reflexive sheaf $\mathcal{E}\otimes\mathcal{L} = (M, \{E^\rho_{\bullet-a_\rho}\}_\rho)$. Here, $E^\rho_{\bullet - a_\rho}$ is the chain of flats whose $j$th term is $E^\rho_{j- a_\rho}$ (cf. \cite[Definition 6.7]{khan2024tropical}). 
	
\begin{rmk}\label{rmk: deg const on iso classes}
	Example 7.10 of \cite{khan2024tropical} establishes the identity
	\begin{equation}\label{eqn: trs deg tensor}
		\deg(\mathcal{E}\otimes \mathcal{L}) = \deg(\mathcal{E}) + \rk(\underline{M})\cdot\deg(\mathcal{L}),
	\end{equation}
	and Remark 7.6 of \cite{khan2024tropical} establishes that the degree of a toric reflexive sheaf is preserved under tropicalization. This establishes that for a complete toric variety $X_\Sigma$ with polarization $h = (c,c,\cdots,c)\in\mathbb{Z}^{\Sigma(1)}$, tensoring with the tropicalization of a principal divisor preserves the degree of a tropical reflexive sheaf.
\end{rmk}

The existence and uniqueness of Harder-Narasimhan filtrations for tropical toric reflexive sheaves rely on certain modularity conditions on flats of the associated $\mathbb{T}$-matroid. Say that a tropical toric reflexive sheaf is \emph{modular} when its underlying $\mathbb{T}$-matroid is. 

In Section \ref{sec: trs proto ex} we define a category $\MTRS$ of modular tropical toric reflexive sheaves and, by using categorical techniques developed in \cite{andre2009slope} and \cite{li2023categorification}, illustrate that the above rank and degree functions define a unique Harder-Narasimhan filtration.

\subsection{Proto-exact categories}\label{subsec: proto ex cat}

In this subsection, we recall the notion of proto-exact categories, introduced by Dyckerhoff and Kapranov \cite{dyckerhoff2019higher}.

\begin{mydef}\label{definition: proto_exact}
	A \emph{proto-exact} category is a category $\mc{C}$ with zero object equipped with two distinguished classes $\MM$ and $\EE$ of morphisms, called \emph{admissible monomorphisms} and \emph{admissible epimorphisms}, respectively. We denote morphisms in $\MM$  by $\rightarrowtail$ and morphisms in $\EE$ by $\twoheadrightarrow$. The classes $\MM$ and $\EE$ are required to satisfy the following conditions:
	\begin{enumerate}
		\item[(PE1)] Every morphism $0 \rightarrow M$ is in $\MM$, and every morphism $M \rightarrow 0$ is in $\EE$.
		\item[(PE2)] The classes $\MM$ and $\EE$ are closed under composition and contain all isomorphisms.
		\item[(PE3)] A commutative square of the following form is Cartesian if and only if it is co-Cartesian:
		\begin{equation}\label{eq: biCartesian_1}
			\begin{tikzcd}
				M\ar[r,tail]\ar[d,two heads,swap]& N
				\ar[d,two heads]\\
				M'\ar[r,tail]& N'
			\end{tikzcd}
		\end{equation}
		\item[(PE4)] Every diagram \(
		\begin{tikzcd}
			M'\ar[r,tail]&N'&N\ar[l, two heads,swap]
		\end{tikzcd}
		\) can be completed to a bi-Cartesian square \eqref{eq: biCartesian_1}.
		\item[(PE5)] Every diagram \(
		\begin{tikzcd}
			M'&M\ar[l,two heads,swap]\ar[r,tail]&N
		\end{tikzcd}
		\) can be completed to a bi-Cartesian square \eqref{eq: biCartesian_1}.
	\end{enumerate} 
\end{mydef}

In a proto-exact category, a commutative diagram of the following form is said to be an admissible short exact sequence:
	\begin{equation}\label{eq: biCartesian_2}
	\begin{tikzcd}
		A\ar[r,tail]\ar[d,two heads,swap]& B
		\ar[d,two heads]\\
		0\ar[r,tail]& C
	\end{tikzcd}
\end{equation}
We denote a short exact sequence by $A \rightarrowtail B \twoheadrightarrow C$. In this case, we also write $C=B/A$ and call $C$ a quotient object.  

Two admissible monomorphisms $f: A \rightarrowtail B$ and $g:A' \rightarrowtail B$ are isomorphic if there exists an isomorphism $\phi:A \to A'$ such that $g\circ \phi = f$. The isomorphism classes in $\mathfrak{M}$ are said to be \emph{admissible subobjects}. We denote admissible subobjects by $A\subseteq B$. 

\begin{myeg} $\;\;\;$
	
\begin{enumerate}
	\item[(1)] Any Quillen exact category is also proto-exact, with the same classes of admissible monomorphisms and epimorphisms. In particular, any abelian category is proto-exact, with $\MM$ consisting of all monomorphisms and $\EE$ consisting of all epimorphisms. 
	
	\item[(2)] The category $\textbf{Set}_\bullet$ of pointed sets and basepoint-preserving functions is proto-exact with $\MM$ consisting of all pointed injections and $\EE$ consisting of all pointed surjections. The full subcategory of finite pointed sets is also proto-exact, as is the subcategory $\textbf{Vect}_{\mathbb{F}_1}$ of $\mathbb{F}_1$-vector spaces, which finite pointed sets with $\mathbb{F}_1$-linear functions. A function $f$ is $\mathbb{F}_1$-linear if $f^{-1}(x)$ has cardinality 1 whenever $x$ is not the basepoint. In the latter two cases, $\MM$ and $\EE$ consist of the restrictions of admissibles in $\textbf{Set}_\bullet$ to these subcategories. 
	
	\item[(3)] The category $\Mat$ of pointed matroids is finitary and proto-exact, as established in \cite{eppolito2018proto}. 
\end{enumerate}
\end{myeg}  

\subsection{Duality and Direct Sum for Proto-Exact Categories}\label{direct_sum_duality}

The proto-exact categories which we encounter in the sequel possess additional structures which behave well with respect to exact sequences. In this section we describe these structures, namely direct sums and duality, in the context of proto-exact categories. For a more detailed account of these concepts, see \cite{eberhardt2022group}.

\begin{mydef}\label{def:pr_ex_dir_sum}
	Let $\mathcal{C}$ be a proto-exact category with classes $\MM$ and $\EE$ of admissible monomorphisms and epimorphisms, respectively. Say that $\mathcal{C}$ has \emph{exact direct sum} if there is a symmetric monoidal product $\oplus$ satisfying the following properties:
	\begin{enumerate}
		\item[(DS1)] The monoidal unit is the zero object 0.
		\item[(DS2)] $\oplus:\mathcal{C}\times\mathcal{C}\to\mathcal{C}$ is an exact functor.
	\end{enumerate}
	For $U,V\in\mathcal{C}$, denote by $\iota_U:U\to U\oplus V$ and $\pi_U:U\oplus V\to U$ the corresponding inclusion and projection, and similar for $V$. 
	\begin{enumerate}
		\item[(DS3)] The following maps are injections:
		\begin{align*}
			\text{Hom}(U\oplus V,W) &\longrightarrow \text{Hom}(U,W)\times\text{Hom}(V,W) \\
			f &\longmapsto (f\circ \iota_U,f\circ \iota_V) \\
			\text{Hom}(W,U\oplus V) &\longrightarrow \text{Hom}(W,U)\times \text{Hom}(W,V) \\
			f &\longrightarrow (\pi_U\circ f,\pi_V\circ f)
		\end{align*}
		\item[(DS4)] Let $U\xrightarrow{\iota} X \xrightarrow{\pi} V$ be a short exact sequence. Then for all sections $s$ of $\pi$, there is a unique isomorphism $\varphi$ for which  have the diagram
		\[
		\begin{tikzcd} 
			& X & \\
			U \ar[ur, tail, "\iota"] \ar[r, tail, "\iota_U", swap] & U\oplus V \ar[u, "\varphi"] & V \ar[ul, "s", swap] \ar[l, tail, "\iota_V"]
		\end{tikzcd}  
		\]
		Further, for all retracts $r$ of $\iota$, there is a unique isomorphism $\psi$ for which we have the diagram
		\[
		\begin{tikzcd}
			& X \ar[dl, "r", swap] \ar[dr, two heads, "\pi"] & \\
			U & U\oplus V \ar[l, two heads, "\pi_U"] \ar[u, "\psi"] \ar[r, two heads, "\pi_V", swap] & V
		\end{tikzcd}
		\]
	\end{enumerate} 
\end{mydef}

We say that a proto-exact category $\mathcal{C}$ with exact direct sum is \emph{combinatorial} if, for any admissible monomorphism $g:X\rightarrow U_1\oplus U_2$ we have admissible subobjects $\iota_k:X_k\rightarrowtail U_k$ for $k=1,2$ for which there is an isomorphism $f:X\rightarrow X_1\oplus X_2$ satisfying $(\iota_1\oplus \iota_2)\circ f = g$. Additionally, we say that $\mathcal{C}$ is \emph{uniquely split} if every short exact sequence splits in a unique way. 

\begin{mydef}\label{def:pr_ex_duality}
	Let $\mathcal{C}$ be a proto-exact category. Say that $\mathcal{C}$ is a proto-exact-category \emph{with duality} if there is a functor $P:\mathcal{C}\to\mathcal{C}^{\text{op}}$ and a natural isomorphism $\Theta:\text{id}_{\mathcal{C}}\to P\circ P^{\text{op}}$ such that for all $U\in\mathcal{C}$ we have $P(\Theta_U)\circ \Theta_{P(U)} = \text{id}_{P(U)}$ and for which $P$ and $\Theta$ satisfy the following properties.
	\begin{enumerate}
		\item[(D1)] $P(0)\cong 0$ 
		\item[(D2)] $f\in\MM$ if and only if $P(f)\in\EE$
		\item[(D3)] If either of the following squares is biCartesan, then the other is as well. 
		\[
		\begin{tikzcd}
			U \ar[r, tail] \ar[d, two heads] & X \ar[d, two heads] & & P(U) & P(X) \ar[l, two heads] \\
			W \ar[r, tail] & V & & P(W) \ar[u, tail] & P(V) \ar[l, two heads] \ar[u, tail]
		\end{tikzcd}
		\]
	\end{enumerate}
\end{mydef}

\subsection{Hall algebras of finitary proto-exact categories} 

Given an abelian category $\mathcal{A}$ satisfying suitable conditions, one can associate to it the Hall algebra $H_\mathcal{A}$. Whether $H_\mathcal{A}$ carries a natural Hopf algebra structure depends delicately on the properties of $\mathcal{A}$. In his celebrated work \cite{kapranov1997eisenstein}, Kapranov studied the case where $\mathcal{A}$ is the category of coherent sheaves $\mathrm{Coh}(X)$ on a smooth projective curve $X$ defined over a finite field $\mathbb{F}q$. He showed that when $X=\mathbb{P}^1$, a certain subalgebra of $H_{\mathrm{Coh}(X)}$ is isomorphic to the positive part of the quantum affine algebra $U_q(\widehat{\mathfrak{sl}_2})$.

The definition of $H_\mathcal{A}$ can be naturally generalized to the case of proto-exact categories $\mathcal{C}$, and one may study the proto-exact structure of $\mathcal{C}$ through the lens of the associated Hall algebra $H_\mathcal{C}$. See \cite{jun2020toric} for some backgrounds on Hall algebras of proto-exact categories along with a combinatorial approach to study the Hall algebra of the category of coherent sheaves on $\mathbb{P}^2$ whose structure is rarely known.  

Now, we recall the definition of $H_\mathcal{C}$. Say that a proto-exact category $\mathcal{C}$ is \emph{finitary} if for all objects $M$ and $N$ of $\mathcal{C}$ we have that the sets $\text{Hom}(M,N)$ and $\text{Ext}^1(M,N)$ are finite.\footnote{For a proto-exact cateogry, as in the classical setting, $\text{Ext}^1(M,N)$ is defined as the set of equivalence classes of admissible short exact sequences. } Given a finitary proto-exact category $\mathcal{C}$ and a field $k$ of characteristic zero, the \emph{Hall algebra} $H_\mc{C}$ of $\mathcal{C}$ over $k$ is defined to be 
	\[H_\mc{C} := \{f: \text{Iso}(\mc{C})\to k\;|\; f\text{ has finite support}\}.\]
Here, $\text{Iso}(\mc{C})$ is the set of isomorphism classes of objects of $\mc{C}$. The Hall algebra is an associative $k$-algebra with product given by
\begin{equation}
(f*g)([M]) = \sum_{N\subseteq M} f([M/N])\cdot g([N]).
\end{equation}
The Hall algebra $H_\mc{C}$ is equipped with a basis of functions $\{\delta_{[M]}\;|\;[M]\in\text{Iso}(\mc{C})\}$, defined in analogy with Dirac measures, given by
\begin{equation}
	\delta_{[M]}([N]) = \begin{cases} 1 & M\cong N, \\ 0 & \text{else.} \end{cases}
\end{equation}
The function $\delta_{[0]}$ is the unit of $H_\mc{C}$ and multiplication of basis elements is given by
\begin{equation}
	\delta_{[M]}*\delta_{[N]} = \sum_{[R]\in\text{Iso}(\mc{C})} g_{M,N}^R\cdot \delta_{[R]}.
\end{equation}
Here the constants $g_{M,N}^R\in \mathbb{N}$ are given by
\begin{equation}
	g_{M,N}^R = \#\{R'\;|\;R'\subseteq R,\;R'\cong N,\;R/R'\cong M\}.
\end{equation}
Note that finitarity ensures that the $g_{M,N}^R$ are natural numbers. When $\mc{C}$ is combinatorial and has a direct sum $\oplus$, the Hall algebra $H_\mc{C}$ can be promoted to a bialgebra by equipping it with \emph{Green's comultiplication}, which is given by
	\begin{align*}
		\Delta:H_\mc{C} &\longrightarrow H_\mc{C}\otimes H_\mc{C} \\
		\Delta(f)([M],[N]) &:= f([M\oplus N]).
	\end{align*} 

\subsection{Proto-abelian categories}\label{subsec: proto ab cat}

In this subsection, we recall the notion of proto-abelian categories, introduced by Andr\'e \cite{andre2009slope}. Given a category $\mathcal{C}$ with zero object, a \emph{kernel} of a morphism $f:M\to N$ is a morphism $k:K\to M$ for which the following square is Cartesian:
		\[
	\begin{tikzcd}
		K\ar[r,"k"]\ar[d]& M
		\ar[d,"f"]\\
		0\ar[r]& M
	\end{tikzcd}
		\]
Dually, a \emph{cokernel} of $f$ is a morphism $c:N\to C$ for which the following square is co-Cartesian:
		\[
	\begin{tikzcd}
		M\ar[r,"f"]\ar[d]& N
		\ar[d,"c"]\\
		0\ar[r]& C
	\end{tikzcd}
		\]
A kernel is necessarily a monomorphism and a cokernel is necessarily an epimorphism. Say that a morphism in $\mathcal{C}$ is a \emph{strict monomorphism} if it is the kernel of another morphism in $\mathcal{C}$. Likewise, a \emph{strict epimorphism} is by definition the cokernel of another morphism. If $M\to N$ is a strict monomorphism (resp. epimorphism), say that $M$ is a \emph{strict subobject} of $N$ (resp. $N$ is a \emph{strict quotient} of $M$). For the duration of this section, we denote strict monomorphisms by $\rightarrowtail$ and strict epimorphisms by $\twoheadrightarrow$. 

By abuse of terminology, we refer to the domain of a kernel of $f$ as a kernel of $f$ and refer to the codomain of a cokernel of $f$ as a cokernel of $f$. These objects are unique up to isomorphism, and it will always be clear from context whether we are speaking of a morphism or an object.

For a category $\mathcal{C}$ with kernels and cokernels, given an object $M$ and two strict subobjects $N_1$ and $N_2$ of $M$ we can construct a \emph{sum} $N_1 + N_2$ and \emph{intersection} $N_1\cap N_2$ of $N_1$ and $N_2$ under $M$. These objects are again strict subobjects of $M$ and are constructed as follows. 

The intersection $N_1\cap N_2$ is simply the pull-back of the diagram $N_1\rightarrowtail M \leftarrowtail N_2$, which exists in $\mathcal{C}$ by \cite[p.~9]{andre2009slope}. Let $M/N_1$ and $M/N_2$ be the cokernels of the inclusions $N_i\rightarrowtail M$. Then one forms the pushout $M'$ of $M/N_1\twoheadleftarrow M \twoheadrightarrow M/N_2$, which is again guaranteed to exist by \cite[p.~9]{andre2009slope}. The \emph{sum} of $N_1$ and $N_2$ is the kernel $N_1 + N_2$ of the resulting strict epimorphism $M\to M'$. Diagrammatically, this can be written as follows:
\begin{equation}\label{diag: sum of subobj}
	\begin{tikzcd} 
		N_1 + N_2 \ar[dr, tail] & N_1 \ar[l, tail] \ar[d, tail] & \\
		N_2 \ar[u, tail] \ar[r, tail] & M \ar[r, two heads] \ar[d, two heads] & M/N_2 \ar[d] \\
		& M/N_1 \ar[r] & M'
	\end{tikzcd}
\end{equation}

It is shown in \cite{li2023categorification} that $N_1$ and $N_2$ are strict subobjects of $N_1 + N_2$ and that $N_1 \cap N_2$ is a strict subobject of both $N_1$ and $N_2$, so that we can take quotients $(N_1 + N_2)/N_i$ and $N_i/(N_1 \cap N_2)$ for $i=1,2$. By \cite[Proposition 3.3]{li2023categorification}, the sum and intersection of subobjects are related by the following Cartesian square:
\begin{equation}\label{diag: Cart sum prod}
	\begin{tikzcd}
		N_1 \cap N_2 \ar[r, tail] \ar[d, tail] & N_1 \ar[d, tail]\\
		N_2 \ar[r, tail] & N_1 + N_2
	\end{tikzcd}
\end{equation} 

\begin{mydef}\label{definition: Andre's proto-abelian}
	A \emph{proto-abelian} category is a pointed category $\mathcal{C}$ which has all kernels and cokernels and additionally satisfies the following conditions:
	\begin{enumerate}
		\item[(PA1)] Any morphism with zero kernel (resp. cokernel) is a monomorphism (resp. epimorphism)
		\item[(PA2)] The pullback of a strict epi along a strict mono is a strict epi, and the pushout of a strict mono along a strict epi is a strict mono. 
	\end{enumerate}
\end{mydef}

Each proto-abelian category carries a natural proto-exact structure, with $\MM$ consisting of strict monomorphisms and $\EE$ of strict epimorphisms. In proto-abelian categories, a \emph{short exact sequence} is a diagram of the form (\ref{eq: biCartesian_2}), where vertical arrows are strict epimorphisms and horizontal arrows are strict monomorphisms. 

\begin{rmk}
	Proto-abelian categories are sometimes defined in a different but related way to the above (see \cite{dyckerhoff2019higher}, \cite{jun2020toric}). In those treatments, a category is said to be proto-abelian if it is proto-exact and the classes $\MM$ and $\EE$ consist of all monomorphisms and all epimorphisms, respectively. Such categories are not necessarily proto-abelian in our sense, and neither notion of proto-abelianness immediately implies the other. 
\end{rmk}

\begin{myeg} $\;\;\;$
	
	\begin{enumerate}
		\item[(1)] 
		Any abelian category is proto-abelian. Conversely, if a proto-abelian category has all finite products and coproducts, and if any morphism which is both monic and epic is an isomorphism, then the category is abelian (see \cite{andre2009slope}).
		
		\item[(2)] The categories of finite-dimensional Euclidean and Hermitian vector spaces with linear maps of norm at most 1 is proto-abelian. This category has all finite coproducts and pushouts, but does not have all products or pullbacks. 
		
		\item[(3)] The category of Euclidean lattices with additive maps of norm at most 1 is proto-abelian (see \cite{andre2009slope}, \cite{li2023categorification}). 
	\end{enumerate} 
\end{myeg} 

We can also promote a proto-exact category to a proto-abelian category, so long as it contains kernels and cokernels and satisfies the first axiom of proto-abelian categories above. 

\begin{pro}\label{prop:protoex_then_protoab}
	Let $\mathcal{C}$ be a category with kernels and cokernels satisfying axiom \emph{(PA1)}. Suppose that $\mathcal{C}$ is proto-exact with $\MM$ being compositions of isomorphisms with strict monomorphisms and $\EE$ being the collection of strict epimorphisms composed with isomorphisms in $\mathcal{C}$. Then $\mathcal{C}$ is proto-abelian. 
\end{pro}
\begin{proof}
	By the definition of $\MM$ and $\EE$ we see that axioms (PE4) and (PE5) of a proto-exact category immediatly ensure that axiom (PA2) holds.  
\end{proof}

\subsection{Stability and slope in proto-abelian categories}\label{subsec: stability and slopes} 

Proto-abelian categories possess sufficient structure to define useful filtrations of their objects using slope functions, generalizing for example the Harder-Narasimhan filtration for vector bundles on curves. See \cite{andre2009slope}, \cite{rudakov1997stability}, and \cite{li2023categorification} for a broader account. We now introduce that part of this theory which will be of use to us in the sequel. 

Suppose we are given \emph{rank} and \emph{degree} functions which are constant on isomorphism classes and additive on short exact sequences, so that they give additive maps out of the Grothendieck group of $\mc{C}$:
	\begin{align*}
		\text{rk}&:K_0(\mathcal{C})\to\mathbb{N} \\
		\text{deg}&:K_0(\mathcal{C})\to\mathbb{R}
	\end{align*} 
Suppose also that $\rk(M)=0$ implies that $M$ is the zero object. We then define a \emph{slope} function $\mu:\text{Iso}(\mathcal{C})\setminus\{0\}\to\mathbb{R}$ on $\mathcal{C}$ by the formula
	\[\mu(M) = \frac{\deg(M)}{\text{rk}(M)}\] 
Such a function $\mu$ is said to satisfy the \emph{strong slope inequality} if for any object $M$ and any strict subobjects $N_1$ and $N_2$ of $M$ we have
\begin{equation} \label{eqn:strong slope ineq}
	\mu\left(\frac{N_1}{N_1\cap N_2}\right) \leq \mu\left(\frac{(N_1 + N_2)}{N_2}\right).
\end{equation} 

Say that an object $M$ of $\mathcal{C}$ is \emph{semistable} (resp.~\emph{stable}) if we have $\mu(N)\leq \mu(M)$ (resp.~$\mu(N)<\mu(M)$) for every strict subobject $N\rightarrowtail M$. 

\begin{mydef}\label{def:slope filtration}
	A \emph{slope filtration} of $M$ is a flag of strict subobjects
\begin{equation}
	0 = N_0\rightarrowtail N_1 \rightarrowtail \cdots \rightarrowtail N_r = M
\end{equation}
such that each quotient $N_{i+1}/N_i$ is semistable and for which $\mu(N_i/N_{i-1})>\mu(N_{i+1}/N_i)$. A \emph{slope filtration on} $\mathcal{C}$ is a functorial assignment of a slope filtration to every object of $\mathcal{C}$. 
\end{mydef}

\begin{rmk}
The paper \cite{andre2009slope} establishes a framework in which slope functions are used to define slope filtrations on proto-abelian categories. The slope functions there are required to be non-decreasing under morphisms which are both monic and epic. Later, this framework was generalized by \cite{li2023categorification} to include the broader class of slope functions satisfying the strong slope inequality (\ref{eqn:strong slope ineq}). The theory in \cite{li2023categorification} also departs from \cite{andre2009slope} in a handful of other subtle ways. However, since we do not need the level of generality attained there, we proceed with slope functions as defined above.
\end{rmk}

An object $M$ of $\mathcal{C}$ is said to be \emph{Noetherian} (resp.~\emph{Artinian}) if it satisfies the Noetherian (resp.~Artinian) property on chains of strict subobjects of $M$. In \cite{li2023categorification}, Li establishes that given a function $\mu: \text{Iso}(\mathcal{C})\setminus\{0\}\to\mathbb{R}$ which satisfies the strong slope inequality, one obtains for each nonzero Noetherian and Artinian object of $\mathcal{C}$ a unique slope filtration associated to $\mu$ as above (see \cite{li2023categorification} Theorems 5.2, 5.3, 5.4, and the discussion in Section 5.3). 

\begin{mythm}[\cite{li2023categorification}, Theorems 5.2, 5.3, 5.4]\label{thm:li HN filtr}
	Let $\mathcal{C}$ be a small proto-abelian category in which all objects are both Artinian and Noetherian and let $\mu:\text{Iso}(\mathcal{C})\setminus\{0\}\to \mathbb{R}$ be a function satisfying the strong slope inequality (\ref{eqn:strong slope ineq}). Then for each nonzero object $M$ of $\mathcal{C}$ there exists a filtration
		\[0 = N_0\rightarrowtail N_1 \rightarrowtail \cdots \rightarrowtail N_r = M\]
	by strict subobjects, unique up to isomorphism, such that for all $i=1,\ldots,r$ we have that $N_i/N_{i-1}$ is semistable and satisfy
		\[\mu(N_1/N_0)>\mu(N_2/N_1)>\cdots>\mu(N_r/N_{r-1}).\]
\end{mythm} 

We call any filtration obtained in this way a \emph{categorical slope filtration}. This establishes in particular that functions satisfying the strong slope inequality are a suitable generalization of the slope functions of \cite{andre2009slope} for the purpose of constructing filtrations. 

In Section \ref{sec:semistab for mtrs}, we illustrate that the slope function of \cite{khan2024tropical}, when restricted to the category $\MTRS$ of modular tropical reflexive sheaves, satisfies the strong slope inequality. This establishes Corollary \ref{cor: HN filtr coincide}, illustrating in particular that the Harder-Narasimhan filtration of \cite{khan2024tropical} is a categorical slope filtration as in Theorem \ref{thm:li HN filtr}.

\section{The category $F$-$\Mat$}\label{sec:category_FMat}

\subsection{Morphisms of Pointed $F$-Matroids}\label{subsec:Fmat_morphisms}

Fix a perfect idyll $F$. Hereafter, by a \emph{matroid} we mean a pointed $F$-matroid unless otherwise indicated. We now recall morphisms of matroids over a perfect idyll $F$, following definitions introduced in \cite{jarra2024quiver}. 

Let $S=\tilde{S}\sqcup\{*_S\}$ and $T=\tilde{T}\sqcup\{*_T\}$ be pointed sets. A \emph{submonomial matrix} over $F$ is a matrix $A$ with entries in $F$ such that each row and column contains at most one nonzero entry. Given such a matrix with rows and columns indexed by $\tilde{T}$ and $\tilde{S}$, respectively, submonomiality ensures that each such $A$ defines a mapping $A:F^{\tilde{S}}\to F^{\tilde{T}}$ via matrix-vector multiplication. The \emph{underlying map} of a submonomial matrix $A$ is the function $\underline{A}:S\to T$ defined by 
\[\underline{A}(i) = \begin{cases} j & i\neq *_S\text{ and }A_{i,j} \neq 0_F \\ *_T & \text{else} \end{cases}.\]
In particular, $\underline{A}$ is an $\mathbb{F}_1$-linear map of pointed sets. It is clear from the definitions that the transpose $A^t$ of such a submonomial matrix $A$ is again submonomial and defines a map from $F^{\tilde{T}}$ to $F^{\tilde{S}}$ and an associated underlying map $\underline{A^t}:T\to S$. We refer to $A^t$ as the \emph{adjoint} or \emph{dual} of $A$.

\begin{mydef}\label{definition: morphism F-matroids}
	A \emph{morphism} of (nonzero) $F$-matroids $f:M\to N$ is a submonomial matrix $f$ over $F$ indexed by $\tilde{E}_N$ and $\tilde{E}_M$ such that $f\cdot\mathcal{V}_M\subseteq\mathcal{V}_N$, where $\mathcal{V}_M$ and $\mathcal{V}_N$ are the sets of vectors of $M$ and $N$, respectively. 
\end{mydef}

\begin{rmk}\label{rmk: zero morphisms}
	Recall that the matroid $U^0_1$ consisting of only a basepoint is indeed a pointed $F$-matroid. $F$-matroid morphisms $M\to N$ where either $N=U^0_1$ or $M=U^0_1$ do not make sense in the above definition, since they would be defined as empty matrices. Instead, note that in either of these cases we have natural notions of morphisms given by the only existing pointed maps on underlying pointed sets. It easily follows that $0:=U^0_1$ is the zero object of $\FMat$.
\end{rmk} 

$F$-matroid morphisms are preserved under duality. Explicitly, a submonomial matrix $f$ is a morphism from $M$ to $N$ if and only if $f^t$ is a morphism from $N^*$ to $M^*$ (\cite[Corollary 2.9]{jarra2024quiver}). In the sequel, we will use the following alternative characterizations of matroid morphisms.

\begin{pro}[\cite{jarra2024quiver}, Theorems 2.8]\label{prop:matroid morphisms}
	Let $M=[\mu]$ and $N=[\nu]$ be $F$-matroids of ranks $m$ and $n$, respectively, and let $f$ be a submonomial $F$-matrix indexed by $\tilde{E}_N\times \tilde{E}_M$. Then the following are equivalent:
	\begin{enumerate}
		\item $f$ is a morphisms from $M$ to $N$ 
		\item for all $\mathbf{x}=(x_0,\ldots,x_m)\in \tilde{E}_M^{m+1}$ and all $\mathbf{y}=(y_2,\ldots,y_n)\in \tilde{E}_N^{n-1}$, we have
		\[\sum_{k=0}^m (-1)^k\cdot\mu(\mathbf{x}_{\hat{k}})\cdot f_{\underline{f}(x_k),x_k}\cdot\nu(\underline{f}(x_k),\mathbf{y})\in N_F.\] 
	\end{enumerate}
\end{pro} 

\begin{rmk}
	In \cite{jarra2024quiver}, Proposition \ref{prop:matroid morphisms} is stated and proven in terms of morphisms of non-pointed matroids. The adaptation of this result to the pointed case is immediate and has been omitted. 
\end{rmk}

For each perfect idyll $F$, the collection of (pointed) $F$-matroids together with the morphisms above forms a category $F$-$\Mat$. One can also define a category $F$-\textbf{Mat} of non-pointed $F$-matroids, but we do not study this category here.

\begin{rmk}\label{rmk:base_change_functor}
	Recall that to each $F$-matroid $M$ is associated an underlying matroid $\underline{M}$. With the above definition of matroid morphisms, the assignment $M\mapsto \underline{M}$ is a functor from $F$-matroids to $\mathbb{K}$-matroids. More generally, for any morphism of perfect idylls $\Phi:F\to G$ we have that the base-change map $M\to \Phi_*M$ is a functor from $F$-$\Mat$ to $G$-$\Mat$. On morphisms, these functors assign to an $F$-matroid morphism $f:M\to N$ a $G$-matroid morphism $\Phi_*f:\Phi_*M\to\Phi_*N$ given by applying $\Phi$ entrywise to the submonomial matrix $f$. 
\end{rmk}

We now illustrate how the contraction, deletion, and restriction of $F$-matroids described in Subsection \ref{subsec:F matroids} can be seen as $F$-matroid morphisms. Let $M$ be an $F$-matroid and let $S\subseteq E_M$.  
\begin{enumerate}
	\item[(1)] (Contraction) Associated to the contraction $M/S$ we have a \emph{contraction map} $c_S:M\to M/S$ given by the submonomial matrix indexed by $\tilde{E}_M\setminus S$ and $\tilde{E}$ with entries $(c_S)_{ij}:= \delta_{ij}$,	where $\delta_{ij}$ is the Kronecker delta. In other words, up to reordering of columns, $c_S$ is a matrix whose first $|\tilde{E}_M\setminus S|$ columns form an identity matrix and whose other columns are zero. That this is a matroid morphism is guaranteed for example by \cite[Proposition 4.4]{anderson2019vectors}.
	\item[(2)] (Restriction) Associated to the restriction $M|S$ we have a \emph{restriction map} $r_S:M|S\to M$ given by the submonomial matrix indexed by $\tilde{E}_M$ and $S$ with entries $(r_S)_{ij} = \delta_{ij}$. Thus in particular we have $r_S = c_{E_M\setminus S}^*$ and so $r_S$ is also a matroid morphism.
\end{enumerate}
These morphisms can be understood in terms of their underlying maps: the contraction $\underline{c_S}$ sends all elements of $S$ to the basepoint of $M\setminus S$ and preserves all other elements, while the restriction $\underline{r_S}$ embeds $S$ into $E_M$.  

Morphisms of $\mathbb{K}$-matroids induce $\mathbb{F}_1$-linear strong maps between the associated classical matroids, in the sense that the underlying map is an $\mathbb{F}_1$-linear strong map of matroids. Similarly, morphisms of $F$-matroids provide an appropriate notion of $\mathbb{F}_1$-linear strong maps for $F=\mathbb{S},\mathbb{T}$.

We now establish some useful facts about morphisms of matroids which will be of use to us in the sequel. In particular, we are able to characterize monomorphisms and epimorphisms completely.

\begin{pro}\label{prop:inj_surj}
	A morphism $f:M\to N$ of $F$-matroids is a monomorphism if and only if $\underline{f}$ is injective. Similarly, $f$ is an epimorphism if and only if $\underline{f}$ is surjective.
\end{pro}
\begin{proof}
	This is established by \cite[Corollary 3.8]{heunen2017category}. Its analog for pointed $F$-matroids is immediate.
\end{proof}

Note that this ensures that $\FMat$ satisfies axiom (PA1) of a proto-abelian category (see Definition \ref{definition: Andre's proto-abelian}). 

\begin{rmk}\label{rmk:epi-mono}
From this we see that any epi-monic $f:M\to N$ of $F$-matroids has underlying map a bijection, but the inverse map on sets is not necessarily a morphism of $F$-matroids. Instead, we see that $f:M\to N$ is an isomorphism of $F$-matroids if and only if the submonomial matrix associated to $f$ is invertible and the inverse matrix $f^{-1}$ defines an $F$-matroid morphism from $N$ to $M$.	
\end{rmk}

\begin{pro}\label{prop:adj_isos}
	A morphism $\alpha:M\to N$ is a matroid isomorphism if and only if its adjoint $\alpha^t:N^*\to M^*$ is also an isomorphism. 
\end{pro}
\begin{proof}
	Considering the $F$-matrices associated to $\alpha$ and to $\alpha^{-1}$, note that 
	\[(\alpha^{-1})_{ij} = \begin{cases} \alpha_{ji}^{-1} & \alpha_{ji}\neq 0 \\ 0 & \text{else} \end{cases}.\]
	On the other hand, the morphism $\alpha^t:N^*\to M^*$ by definition has the $F$-matrix given by $(\alpha^t)_{ij} = \alpha_{ji}$ and the adjoint of $\alpha^{-1}$ has the $F$-matrix given by 
	\[((\alpha^{-1})^t)_{ij} = (\alpha^{-1})_{ji} = \begin{cases} \alpha_{ij}^{-1} & \alpha_{ij}\neq 0_F \\ 0_F & \text{else} \end{cases}.\] 
	Thus as $F$-matrices, $(\alpha^{-1})^t = (\alpha^t)^{-1}$. Finally the fact that $\alpha^{-1}:N\to M$ is a morphism implies immediately that $(\alpha^{-1})^t:M^*\to N^*$ is also a morphism, and since $(\alpha^{-1})^t = (\alpha^t)^{-1}$ the result follows.
\end{proof} 

\begin{pro}\label{proposition: restriction-contraction iso}
	Let $\alpha:M\to N$ be a matroid isomorphism, let $A\subseteq E_N$, and $S\subseteq E_M$. Then $\alpha$ induces isomorphisms $M|\alpha^{-1}(A)\cong N|A$ and $M/S\cong N/\alpha(S)$.
\end{pro}
\begin{proof}
	Let $M=[\mu]$ and $N=[\nu]$ be matroids with ranks $m$ and $n$. We begin with the contraction statement. Denote the associated contractions $M/S = [\mu'']$ and $N/\alpha(S) = [\nu'']$. Assume also that $M/S$ has rank $m''$ and $N/\alpha(S)$ has rank $n''$. Then by definition $\alpha$ induces a bijection on the ground sets of $M/S$ and $N/\alpha(S)$. We seek to show that this induced bijection is a matroid morphism. That is, we would like that for all $\mathbf{x}=(x_0,\ldots,x_{m''})\in (E_M\setminus S)^{m''+1}$ and all $\mathbf{y}\in (E_N\setminus \alpha(S))^{n''-1}$ we have
	\begin{equation}\label{eqn:res_con_isos_1}
		\sum_{k=0}^{m''} (-1)^k \mu''(\mathbf{x}_{\hat{k}})\alpha_{\underline{\alpha}(x_k),x_k} \nu''(\underline{\alpha}(x_k),\mathbf{y})\in N_F.
	\end{equation} 
	Since $\alpha$ is an isomorphism, we know that for all $\mathbf{z}\in E_M^{m+1}$ and all $\mathbf{w}\in E_N^{n-1}$ we have
	\begin{equation}\label{eqn:res_con_isos_2}
		\sum_{k=0}^{m} (-1)^k \mu(\mathbf{z}_{\hat{k}}) \alpha_{\underline{\alpha}(z_k),z_k} \nu(\underline{\alpha}(z_k),\mathbf{w}) \in N_F.
	\end{equation} 
	Fix (ordered) bases $\mathbf{a}$ of $\underline{M}|S$ and $\mathbf{d}$ of $\underline{N}|\alpha(S)$. Then taking $\mathbf{z} = (\mathbf{x},\mathbf{a})$ and $\mathbf{w} = (\mathbf{y},\mathbf{d})$, we see that Expressions \eqref{eqn:res_con_isos_1} and \eqref{eqn:res_con_isos_2} differ only by the following terms:
	\begin{equation}\label{eqn:res_con_isos_3}
		\sum_{k>m''+1} (-1)^k \mu(\mathbf{x},\mathbf{a}_{\hat{k}}) \alpha_{\underline{\alpha}(a_k),a_k} \nu(\underline{\alpha}(a_k),\mathbf{y},\mathbf{d}).
	\end{equation}
	Considering each summand in Expression \eqref{eqn:res_con_isos_3}, note that by Proposition \ref{prop:classical_matroid_nonzero_GP} we see that the each factor  $\nu(\underline{\alpha}(a_k),\mathbf{y},\mathbf{d})$ is nonzero only when $|\mathbf{y}|\cup|\mathbf{d}|\cup\{\underline{\alpha}(a_k)\}$ is a basis of $\underline{N}$. But since $\underline{\alpha}(a_k)\in \alpha(S)$ by definition, we see that necessarily $|\mathbf{d}|\cup\{\underline{\alpha}(a_k)\}$ is dependent in $\underline{N}|\alpha(S)$, as $\mathbf{d}$ was chosen to be a basis. This then implies that $|\mathbf{y}|\cup|\mathbf{d}|\cup\{\underline{\alpha}(a_k)\}$ cannot be a basis of $N$ and so $\nu(\underline{\alpha}(a_k),\mathbf{y},\mathbf{d})=0_F$. Thus Expressions \eqref{eqn:res_con_isos_1} and \eqref{eqn:res_con_isos_2} express the same element of $F$ and so $\alpha$ induces a morphism $M/S\to N/\alpha(S)$. Finally, because $\alpha$ is an isomorphism, the above argument applies also to $\alpha^{-1}$, so that $M/S\cong N/\alpha(S)$. 
	
	For the restriction statement, we appeal to duality. In particular note that by Proposition \ref{prop:adj_isos} $\alpha$ induces an isomorphism $\alpha^t:N^*\to M^*$. For $A\subseteq E_N$ the above argument establishes an isomorphism $N^*/A\cong M^*/\alpha^t(A)$, which is equivalent to the desired $N|A\cong M|\alpha^{-1}(A)$ since $\underline{\alpha^t} = \underline{\alpha^{-1}}$.
\end{proof} 

\subsection{(Co)kernels of $F$-matroid morphisms}\label{subsec:(co)kernels}

We now show that each morphism of $F$-matroids admits both a kernel and cokernel, each of which is again a morphism of $F$-matroids. 

\begin{pro}\label{prop:ker and coker FMat}
Let $M=[\mu]$ and $N=[\nu]$ be matroids of ranks $m$ and $n$, respectively, and let $f:M\to N$ be a morphism of $F$-matroids. Then the restriction morphism $M|\underline{f}^{-1}(*_N)\rightarrowtail M$ is a kernel of $f$ and the contraction morphism $N\twoheadrightarrow N/\underline{f}(E_M)$ is a cokernel of $f$. 
\end{pro} 
\begin{proof}
	Let $K:=\underline{f}^{-1}(*_N)\setminus \{*_M\}\subseteq E_M$ and $C:= \underline{f}(\tilde{E}_M)\subseteq E_N$ and suppose that $\rk_{\underline{M}}(K)=k$ and $\rk_{\underline{N}}(C)=c$. Note that we then have the following commutative squares in $F$-$\Mat$:
	\[
	\begin{tikzcd}
		M|K \ar[r, tail] \ar[d, two heads] & M \ar[d, "f"] & & M \ar[r,"f"] \ar[d, two heads]& N \ar[d, two heads] \\
		0 \ar[r, tail] & N & & 0 \ar[r, tail] & N/C
	\end{tikzcd}
	\]
	We seek to show that the first square is Cartesian and the second is co-Cartesian. To this end, suppose that $R=[\rho]$ is a matroid of rank $r$ for which the solid arrows of following diagram commute:
	\begin{equation}
	\begin{tikzcd}\label{diag:kernel}
		R \ar[ddr, bend right=30] \ar[drr, bend left=30, "g"] \ar[dr, "h", dashed]& & \\
		& M|K \ar[d, two heads] \ar[r, "r_K", tail] & M \ar[d, "f"] \\
		& 0 \ar[r, tail] & N  
	\end{tikzcd}
	\end{equation}
	We now construct the morphism $h$ which makes the full diagram commute. Noting that $g$ is given by an $|\tilde{E}_M|\times|\tilde{E}_R|$ submonomial matrix, let $h$ be the $|K|\times |\tilde{E}_R|$ submatrix obtained from $g$ by deleting all rows associated to elements of $\tilde{E}_M\setminus K$.
	
	To see that the resulting matrix $h$ is a matroid morphism, recall first that since $g$ is a matroid morphism, for all $\mathbf{y}\in E_R^{r+1}$ and all $\mathbf{x}\in E_M^{m-1}$ we have
		\[\sum_{i=0}^r (-1)^i\rho(\mathbf{y}_{\hat{i}})\cdot g_{\underline{g}(y_i),y_i}\cdot \mu(\underline{g}(y_i),\mathbf{x})\in N_F.\]
	In order for $h$ to be a morphism from $R$ to $M|K$, we thus need that for each $\mathbf{y}\in E_R^{r+1}$ and each $\mathbf{z}\in K^{\rk_{\underline{M}}(K)-1}$ we have
		\[\sum_{i=0}^r (-1)^i\rho(\mathbf{y}_{\hat{i}})\cdot h_{\underline{h}(y_i),y_i}\cdot (\mu|K)(\underline{h}(y_i),\mathbf{z})\in N_F.\]
	However, recalling the definition of $\mu|K$ we see that we can take $\mathbf{x} = (\mathbf{z},\mathbf{b})$ for any basis $\mathbf{b}$ of $\underline{M}/K$ to obtain 
		\[(\mu|K)(\underline{h}(y_i),\mathbf{z})=\mu(\underline{h}(y_i),\mathbf{x}).\] 
	Moreover, by the construction, for each $y_i\in K$  we have $g_{\underline{g}(y_i),y_i}=h_{\underline{h}(y_i),y_i}$.  In particular, we thus see that  
		\[\sum_{i=0}^r (-1)^i\rho(\mathbf{y}_{\hat{i}})\cdot h_{\underline{h}(y_i),y_i}\cdot (\mu|K)(\underline{h}(y_i),\mathbf{z}) = \sum_{i=0}^r (-1)^i\rho(\mathbf{y}_{\hat{i}})\cdot g_{\underline{g}(y_i),y_i}\cdot \mu(\underline{g}(y_i),\mathbf{x}) \in N_F.\]
	Observe further that any such morphism $h$ must be unique. To see this, note that for any two such morphisms $h$ and $h'$ making the diagram commute, the fact that $r_K$ is a monomorphism implies that $h=h'$. 
	
	The proof that $N/C$ is a cokernel of $f$ is completely analogous and so has been omitted. 
\end{proof}

The above thus illustrates that $F$-$\Mat$ is a pointed category with kernels and cokernels. We have already seen that it satisfies axiom (PA1) of Definition \ref{definition: Andre's proto-abelian}. 

\subsection{Proto-exactness of $F$-$\Mat$}\label{subsec:proto-ex of FMat}

We now prove the main structural result about the category of (pointed) matroids over a perfect idyll $F$. 

\begin{mythm}\label{thm:F mat proto ex}
	For each perfect idyll $F$, the category $F$-$\Mat$ is proto-exact.
\end{mythm}

Combining this theorem with Propositions \ref{prop:protoex_then_protoab}, \ref{prop:inj_surj}, and \ref{prop:ker and coker FMat}, we obtain the following. 

\begin{cor}\label{cor:FMat proto ab}
	For each perfect idyll $F$, $F$-$\Mat$ is proto-abelian. 
\end{cor}

We begin by introducing the classes $\MM$ and $\EE$ of admissible monomorphisms and epimorphisms. 

\begin{mydef}\label{definition: FMat M and E}
	Let $f:M\to N$ be a morphism in $\FMat$. 	
	\begin{enumerate}
		\item 
		$f$ is an admissible monomorphism if $f=r_A \circ \alpha$, where $A \subseteq E_N$ and $r_A:M|A \to M$ is an inclusion map, and $\alpha$ is an isomorphism of $F$-matroids: 
		\[\begin{tikzcd}
			M\ar[r,"\alpha","\sim"'] & N|A\ar[r,"r_A"] & N,
		\end{tikzcd}\]
		Denote by $\MM$ the class of admissible monomorphisms. 
		\item 
		$f$ is an admissible epimorphism if $f=\beta \circ c_B$, where $B \subseteq E_M$ and $c_B:M \to M/B$ is a contraction map and $\beta$ is an isomorphism of $F$-matroids:
		\[\begin{tikzcd}
			M\ar[r,"c_B"] & M/B\ar[r,"\beta","\sim"'] & N.
		\end{tikzcd}\]  
		Denote by $\EE$ the class of admissible epimorphisms. 
	\end{enumerate}
\end{mydef}

By the characterization in the previous section of submonomial matrices corresponding to restrictions and contractions, note that, up to reordering ground set elements, any $f\in \MM$ is represented by an $|\tilde{E}_N|\times|\tilde{E}_M|$ submonomial matrix whose topmost $|\tilde{E}_M|\times|\tilde{E}_M|$ submatrix is invertible (in fact this submatrix is the matrix of $\alpha$). Likewise, up to reordering, any $g\in\EE$ is represented by a submonomial matrix whose leftmost full submatrix is invertibe, and is the matrix of $\beta$. Note that from Proposition \ref{prop:adj_isos}, we have that $\MM^t = \EE$ and vice versa.

Further, by Proposition \ref{prop:ker and coker FMat} we see that $\MM$ consists of isomorphisms followed by kernels and $\EE$ consists of cokernels followed by isomorphisms. 

With the above characterizations of the maps in $\MM$ and $\EE$, we are ready to verify the axioms of a proto-exact category.

\begin{pro}\label{prop:contian iso}
	The classes $\MM$ and $\EE$ contain all isomorphisms. 
\end{pro}
\begin{proof}
	Let $\alpha$ be an isomorphism of matroids. Observe that we can write $\alpha = I\cdot\alpha = \alpha\cdot I$, where $I$ is the identity $F$-matrix. These factorizations and the fact that $\alpha$ is an isomorphism illustrate that $\alpha$ is in both $\MM$ and $\EE$. 
\end{proof}

\begin{pro}\label{prop:composition}
	The classes $\MM$ and $\EE$ are closed under composition. 
\end{pro}
\begin{proof}
	Let $f,g\in \MM$ be a pair of composable admissible monomorphisms, so that we have $f = r_A\circ \alpha:M\to N$ and $g = r_B\circ\beta:N\to R$. Thus we have
	\[\begin{tikzcd}
		M \ar[r,"\alpha","\sim"'] & N|A \ar[r,"r_A"] & N \ar[r,"\beta","\sim"'] & R|B \ar[r,"r_b"] & R.
	\end{tikzcd}\]
	Now define $h:=g\circ f = r_B\circ\beta\circ r_A\circ\alpha$. Then $h:M\to R$ is a $|\tilde{E}_R|\times|\tilde{E}_M|$ $F$-matrix with entries
	\[h_{ij} = \begin{cases} \beta_{\underline{\beta}(\underline{\alpha}(j)),\underline{\alpha}(j)}\cdot \alpha_{\underline{\alpha}(j),j} & i\in \underline{g\circ f}(\tilde{E}_M)\text{ and }i=\underline{g\circ f}(j) \\ 0_F & i\notin \underline{g\circ f}(\tilde{E}_M)\end{cases}.\]
	We can now illustrate the appropriate factorization of $h$ by constructing an isomorphism $\gamma$ by applying $\beta\circ\alpha$ to $\underline{\alpha}^{-1}(A)=E_M$. Define an $|\tilde{E}_M|\times|\tilde{E}_M|$ $F$-matrix $\gamma$ with entries
	\[\gamma_{ij} := \begin{cases} \beta_{\underline{\beta}(\underline{\alpha}(j)),\underline{\alpha}(j)}\cdot \alpha_{\underline{\alpha}(j),j} & i = \underline{g\circ f}(j) \\ 0 & \text{else}\end{cases}.\]
	Thus by construction we have the factorization $h = r_{\underline{g\circ f}(\tilde{E}_M)}\circ\gamma$ as $F$-matrices. Finally, note that $\gamma$ is an isomorphism of matroids since it is a composition of (restrictions of) isomorphisms.
	
	We now show the corresponding result for $\EE$ by appealing to adjoint morphisms. Supposing now that $f,g\in\EE$ are composable admissible epimorphisms with associated factorizations $f = \alpha\circ c_A:M\to N$ and $g = \beta\circ c_B:N\to R$, we can write 
	\[\begin{tikzcd}
		M \ar[r,"c_A"] & M/A \ar[r,"\alpha", "\sim"'] & N \ar[r,"c_B"] & N/B \ar[r,"\beta","\sim"'] & R
	\end{tikzcd}.\]
	Let $h:=g\circ f$. Taking duals of everything in sight, we now have that $g^t:R^*\to N^*$ and $f^t:N^*\to M^*$ are composable morphisms in $\MM$. Thus by the above argument we have that the composition $h^t = f^t\circ g^t$ is also in $\MM$, so that $h\in \EE$. Alternatively, one can easily reproduce an analog of the above argument for $\EE$.
\end{proof}

We now prove an auxiliary result regarding biCartesian squares in $F$-$\Mat$, following \cite[Lemma 4.9]{eppolito2018proto}. 

\begin{lem}\label{lem:aux}
	Let $N$ be a pointed $F$-matroid and let $B\subseteq A\subseteq \tilde{E}_N$. The following square is biCartesian in $F$-$\Mat$:
	\begin{equation}
	\begin{tikzcd}\label{diag:aux_square}
		N|A \ar[r, tail, "f"] \ar[d, two heads, "h"] & N \ar[d, two heads, "k"] \\
		(N|A)/B \ar[r, tail, "g"] & N/B
	\end{tikzcd}
	\end{equation}
where $f = r_A$ and $g = r_{(A\setminus B)}$.
\end{lem}
\begin{proof}
	Note first that this diagram is commutative since by Proposition \ref{prop:restriction-contraction} we have that $(N|A)/B = (N/B)|A$. We proceed by illustrating that the square \eqref{diag:aux_square} is coCartesian. Suppose that we have the following diagram in $F$-$\Mat$:
	\begin{equation}\label{diag:aux_with_R}
		\begin{tikzcd}
			N|A \ar[r,tail,"f"] \ar[d,two heads, "h"] & N \ar[d,two heads, "k"] \ar[ddr,"\beta"] & \\
			(N|A)/B \ar[r,tail,"g"] \ar[drr,"\alpha",swap] & N/B & \\
			& & R
		\end{tikzcd}
	\end{equation}
	Now define an $F$-matrix $\gamma$ given by removing from $\beta$ all columns associated to elements of $B$. Doing so preserves submonomiality and therefore defines a map $\underline{\gamma}:E_{N/B}\to E_R$ and gives a factorization $\gamma k = \beta$ as $F$-matrices. It remains now to show that $\gamma$ is an $F$-matroid map, which we check using Proposition \ref{prop:matroid morphisms}.
	
	Let $M=[\mu]$, $N=[\nu]$, and $R = [\rho]$, and suppose that $\text{rk}_{\underline{N}}(B)=\ell$. In order to show that $\gamma$ is a morphism, we must establish that for all choices of $\mathbf{x}=(x_0,\ldots,x_{r_N-\ell})\in(\tilde{E}_N\setminus B)^{r_N-\ell+1}$ and all choices of $\mathbf{y}=(y_2,\ldots,y_{r_R})\in(\tilde{E}_R)^{r_R-1}$ we have
	\begin{equation}\label{eq: gp}
	\sum_{j=0}^{r_N-\ell}(-1)^j(\nu/B)(\mathbf{x}_{\hat{j}})\gamma_{\underline{\gamma}(x_j),x_j}\rho(\underline{\gamma}(x_j),\mathbf{y})\in N_F,
	\end{equation}
	where $r_N$ is the rank of $N$. We can re-express \eqref{eq: gp} by using the definition of the Grassmann-Pl\"{u}cker function $\nu/B$. Let $\mathbf{b}=(b_1,\ldots,b_\ell)\in (\tilde{E}_N)^\ell$ be a vector where $b_1,\ldots,b_\ell$ is a maximal independent set of $\underline{N}|B$. Then we are interested in the sum
	\begin{equation} \label{eq:gamma_morphism}
		\sum_{j=0}^{r_N-\ell}(-1)^j\nu(\mathbf{x}_{\hat{j}},\mathbf{b})\gamma_{\underline{\gamma}(x_j),x_j}\rho(\underline{\gamma}(x_j),\mathbf{y}).
	\end{equation} 
	However, recalling that $\beta:N\to R$ is a morphism, we are guaranteed that 
	\begin{equation}\label{eq: second sum}
	\sum_{j=0}^{r_N-\ell}(-1)^j \nu(\mathbf{x}_{\hat{j}},\mathbf{b})\beta_{\underline{\beta}(x_j),x_j}\rho(\underline{\beta}(x_j),\mathbf{y}) + \sum_{j=r_N-\ell+1}^{r_N}(-1)^j \nu(\mathbf{x},\mathbf{b}_{\hat{j}})\beta_{\underline{\beta}(b_j),b_j}\rho(\underline{\beta}(b_j),\mathbf{y})\in N_F.
	\end{equation} 
	Noting that $f$ acts as the identity on elements of $B$, we see that by the commutativity of Diagram (\ref{diag:aux_with_R}) we obtain 
	\begin{equation}
	\beta(B) = \beta(f(B)) = \alpha(h(B))=*_R,
	\end{equation}
in particular for all $b_j$ we have $\beta_{\underline{\beta}(b_j),b_j}=0_F$. This illustrates that the second sum in \eqref{eq: second sum} collapses to $0_F$, so that the first sum must be in $N_F$. But by the construction of $\gamma$, for all $x\notin B$ we have $\gamma_{\underline{\gamma}(x),x}=\beta_{\underline{\beta}(x),x}$, hence the first sum is equal to the sum \eqref{eq:gamma_morphism}. This establishes that $\gamma:N/B\to R$ is indeed a morphism of $F$-matroids. 
	
	For the uniqueness of $\gamma$, suppose that $\delta:N/B\to R$ is another morphism making Diagram \eqref{diag:aux_with_R} commute, so that $\beta = \delta k = \gamma k$. By Proposition \ref{prop:inj_surj}, $k$ is an epimorphism, hence $\delta = \gamma$. This establishes that Diagram \eqref{diag:aux_square} is coCartesian.
	
	To see that our diagram is also Cartesian, we could reproduce a similar argument to the above, which involves changing little other than the morphisms involved and the order of some compositions. Instead, we apply duality for a swifter proof. Consider the following pair of dual diagrams:
	\begin{equation}
	\begin{tikzcd}\label{diag:aux_with_P}
		P \ar[ddr,"\alpha"] \ar[drr,"\beta"] & & & (N/B)^* \ar[d,"g^t"] \ar[r,"k^t"] & N^* \ar[d,"f^t",swap] \ar[ddr,"\beta^t"] & \\ 
		& N|A \ar[r,tail,"f",swap] \ar[d,two heads, "h"] & N \ar[d,two heads, "k"] & ((N|A)/B))^* \ar[r,"h^t"] \ar[drr,"\alpha^t"] & (N|A)^* & \\ 
		& (N|A)/B \ar[r,tail,"g"] & N/B & & & P^* \\ 
	\end{tikzcd}
	\end{equation}
	We would like to produce a morphism $P\to N|A$ making the left-hand diagram commute and then prove its uniqueness. In view of duality, this is equivalent to producing such a morphism $(N|A)^*\to P^*$ which commutes with the right-hand diagram. By Proposition \ref{prop:restriction-contraction}, the right-hand diagram above can be re-expressed as follows:
	\[
	\begin{tikzcd}
		N^*|(E_N\setminus B) \ar[r,"k^t"] \ar[d,"g^t"] & N^* \ar[d,"f^t",swap] \ar[ddr,"\beta^t"] & \\
		(N^*|(E_N\setminus B))/(E_N\setminus A) \ar[r,"h^t"] \ar[drr,"\alpha^t"] & N^*/(E_N\setminus A) & \\
		& & P^*	
	\end{tikzcd}
	\]
	But we have just seen in the above that this diagram admits a unique morphism $N^*/(E_N\setminus A)\to P^*$ making the diagram commute. Thus taking duals again we see that Diagram \eqref{diag:aux_square} is Cartesian. 
	
	Alternatively, one can establish directly that Diagram \eqref{diag:aux_square} is Cartesian using a proof totally analogous to the coCartesian case. This approach has the benefit of circumventing duality, and so applies also in the context of simple $F$-matroids. 
\end{proof}

\begin{pro}\label{prop:LR bicart}
	Any diagram in $F$-$\Mat$ of the form $M\rightarrowtail R \twoheadleftarrow N$ can be completed to a bi-Cartesian square.
\end{pro}
\begin{proof}
	Let $i:M \rightarrowtail R$ and $j:N \twoheadrightarrow R$. 
	Using the definition of the class $\EE$, we can write $j=\beta\circ c_A$, where $\beta$ is an isomorphism and $c_A:N \to N/A$ is a contraction for some $A \subseteq E_N$. Likewise, we can write $i=r_B\circ \alpha$, where $\alpha$ is an isomorphism and $r_B:R|B \to R$ is a restriction for some $B \subseteq E_R$. Letting $C:= \underline{j}^{-1}(B)\subseteq E_{N}$, one can see that $A\subseteq C$. Hence, we obtain the following diagram:
	\begin{equation}
		\begin{tikzcd}
			& N|C \ar[ddl] \ar[r, tail] \ar[d, two heads] & N \ar[d,"c_A", two heads] \\
			& (N/A)|C \ar[d,"\beta|_C"] \ar[r, tail] & N/A \ar[d,"\beta"] \\
			M \ar[r,"\alpha"] & R|B \ar[r,"r_B"] & R	
		\end{tikzcd}
	\end{equation}
	Note that Proposition \ref{proposition: restriction-contraction iso} gives that $\beta|_C$ is an isomorphism. Since $\alpha$ is an isomorphism, we have the following commutative square:
	\begin{equation}\label{eq: biCartesian}
		\begin{tikzcd}
			N|C \ar[r, tail] \ar[d, two heads] & N \ar[d, two heads] \\
			M \ar[r,tail] & R 
		\end{tikzcd}
	\end{equation}
	We claim that \eqref{eq: biCartesian} is the bi-Cartesian completion of $M\rightarrowtail R \twoheadleftarrow N$. By Lemma \ref{lem:aux} the square in the following diagram is biCartesian:
	\begin{equation}
		\begin{tikzcd}\label{diag:aux2}
			& N|C \ar[r, tail] \ar[d, two heads] & N \ar[d, two heads] \\ 
			(N|C)/A \ar[r,equal] & (N/A)|C \ar[r,tail] & N/A
		\end{tikzcd}
	\end{equation}
	Here the equality in the bottom left is given by Proposition \ref{prop:restriction-contraction}. Suppose now that we have diagrams 
	\begin{equation}
		\begin{tikzcd}
			P \ar[ddr,"f",swap] \ar[drr] & & & N|C \ar[d,two heads] \ar[r,tail] & N \ar[d,two heads] \ar[ddr] & \\ 
			& N|C \ar[r,tail] \ar[d,two heads] & N \ar[d,two heads] & M \ar[r, tail] \ar[drr,"g",swap] & R & \\ 
			& M \ar[r,tail] & R & & & Q \\ 
		\end{tikzcd}
	\end{equation}
	Then appealing to the fact that $\beta|_C$ is an isomorphism, we now have
	\begin{equation}
		\begin{tikzcd}
			P \ar[ddr,"(\beta|_C)^{-1}\alpha f",swap] \ar[drr] & & & N|C \ar[d,two heads] \ar[r,tail] & N \ar[d,two heads] \ar[ddr] & \\ 
			& N|C \ar[r,tail] \ar[d,two heads] & N \ar[d,two heads] & (N/A)|C \ar[r, tail] \ar[drr,"g\alpha^{-1}(\beta|_C)",swap] & N/A & \\ 
			& (N/A)|C \ar[r,tail] & N/A & & & Q \\ 
		\end{tikzcd}
	\end{equation}
	In particular, the bi-Cartesian property of \eqref{diag:aux2} now guarantees the existence of unique maps $P\to N|C$ and $R\cong N/A \to Q$. This establishes that the completion is biCartesian.
\end{proof}

\begin{pro}\label{prop:UL bicart}
	Any diagram in $F$-$\Mat$ of the form $M\twoheadleftarrow R \rightarrowtail N$ can be completed to a biCartesian square. 
\end{pro}
\begin{proof}
	This proof follows immediately from Proposition \ref{prop:LR bicart} by applying duality as in the proof of Lemma \ref{lem:aux}. However, alternatively one can follow analogous steps to the proof of Proposition \ref{prop:LR bicart} as would be played out in the opposite category. 
\end{proof}

\begin{pro}\label{prop:cart iff cocart}
	Any square of the following form in $F$-$\Mat$ is Cartesian if and only if it is coCartesian:
	\[
	\begin{tikzcd}
		M \ar[r, tail] \ar[d, two heads] & N \ar[d, two heads] \\
		P \ar[r, tail] & R
	\end{tikzcd}
	\]
\end{pro}
\begin{proof}
	Suppose first that the square is Cartesian. Then by Proposition \ref{prop:LR bicart}, $P \rightarrowtail R \twoheadleftarrow N$  can be extended to the following commutative square:
	\[
	\begin{tikzcd}
		M' \ar[d, two heads] \ar[r,tail] & N \ar[d, two heads] \\
		P \ar[r, tail] & R
	\end{tikzcd}
	\]
	Since our original square was assumed Cartesian, we obtain that $M\cong M'$ by the uniqueness of pullbacks up to isomorphism. In particular, this means that our square is coCartesian as well. The other direction is analogous, instead considering the biCartesian completion of $P \twoheadleftarrow M \rightarrowtail N$.  
\end{proof}

Putting together Propositions \ref{prop:contian iso}, \ref{prop:composition},\ref{prop:LR bicart}, \ref{prop:UL bicart} and \ref{prop:cart iff cocart} we obtain Theorem \ref{thm:F mat proto ex} and therefore also Corollary \ref{cor:FMat proto ab}. Observe also that the functor $\Phi_*:\FMat\to G$-$\Mat$ of Remark \ref{rmk:base_change_functor} given by a morphism $\Phi:F\to G$ of perfect idylls is an exact functor, owing to the definition of admissibles. 

We now characterize the sum and intersection of strict subobjects in $F$-$\Mat$ as described in Subsection \ref{subsec: stability and slopes}. Note that from Proposition \ref{prop:ker and coker FMat}, the notion of admissible subobjects in $F$-$\Mat$ (viewed as a proto-exact category) agree with the notion of strict subobjects in $F$-$\Mat$ (viewed as a proto-abelian category). Hence, we use the same symbol $\rightarrowtail$ to denote strict subobjects.

Let $M$ be an $F$-matroid with strict subobjects $M_1,M_2\rightarrowtail M$. Without loss of generality, we write $M_1 = M|A_1$ and $M_2 = M|A_2$ for some $A_1,A_2\subseteq E_M$. In analogy with Diagram \eqref{diag: sum of subobj} we then draw the diagram 
\[
\begin{tikzcd}
	(M|A_1) + (M|A_2) \ar[dr, tail] & M|A_1 \ar[d, tail] \ar[l, tail]  & \\
	M|A_2 \ar[r, tail] \ar[u, tail] & M \ar[d, two heads] \ar[r, two heads] & M/A_2 \ar[d, two heads] \\
	& M/A_1 \ar[r, two heads] & M' 
\end{tikzcd}
\]

where as in Diagram \ref{diag: sum of subobj} the bottom right square is a pushout and the map $(M|A_1) + (M|A_2) \rightarrowtail M$ is the kernel of the composition $M\to M'$. By a straightforward argument similar to the proof of Lemma \ref{lem:aux}, it follows that $M'$ can be identified with the matroid $M/(A_1\cup A_2)$, so that Proposition \ref{prop:ker and coker FMat} yields $(M|A_1)+(M|A_2) = M|(A_1\cup A_2)$. Likewise, one verifies analogously that $(M|A_1)\cap(M|A_2) = M|(A_1\cap A_2)$. This establishes the following proposition. 

\begin{pro}\label{prop: sum prod of subobj FMat}
	Let $M_1,M_2\rightarrowtail M$ be strict subobjects of $M$ in $\FMat$ given as restrictions $M_1= M|A_1$ and $M_2= M|A_2$. Then the sum and intersection of the subobjects $M_1$ and $M_2$ are given by
		\begin{align*}
			(M|A_1)+(M|A_2) &= M|(A_1\cup A_2) \\
			(M|A_1)\cap(M|A_2) &= M|(A_1\cap A_2).
		\end{align*}
\end{pro}

We close this subsection with a verification of Artinian and Noetherian conditions in $\FMat$. 

\begin{pro}\label{pro: art noeth FMat}
	Let $M$ be an $F$-matroid and let $\cdots \rightarrowtail M_{i-1} \rightarrowtail M_i \rightarrowtail \cdots$ be an infinite chain of strict subobjects of $M$. Then for some $n_-,n_+\in\mathbb{Z}$ we have that $M_{i-1}\cong M_i$ for $i\leq n_-$ and $M_j \cong M_{j+1}$ for $j\geq n_+$. 
\end{pro} 
\begin{proof}
	Since each $M_i$ is a strict subobject of $M$, we may write $M_i = M|A_i$ for some $A_i\subseteq E_M$. In particular, since $E_M$ is a finite set and $M_{i-1}\rightarrowtail M_i$ implies $A_{i-1}\subseteq A_i$, we must have that the sets $A_i$ stabilize as $i\to\infty$ or $i\to -\infty$. This completes the proof.  
\end{proof}

\subsection{Further Structure on $F$-$\Mat$}\label{subsec:comb dir sum and duality FMat}

In this subsection we establish directly that the categories $F$-$\Mat$ are combinatorial and have direct sum and duality in the sense of Subsection \ref{direct_sum_duality}. We continue to work with a perfect idyll $F$. 

\begin{mythm}\label{theorem: further structure}
	The category $F$-$\Mat$ is a combinatorial proto-exact category with duality and with exact direct sum. 
\end{mythm}

We prove this theorem in three parts as Propositions \ref{prop:pr_ex_dir_sum}, \ref{prop:pr_ex_duality}, and \ref{prop:comb_dir_sum} which we separate for convenience. The reader may wish to refer back to Definitions \ref{def:pr_ex_dir_sum} and \ref{def:pr_ex_duality}. 

\begin{lem}\label{lem:ctr_del_dir_sum}
	Let $M_1=[\mu_1]$ and $M_2=[\mu_2]$ be pointed $F$-matroids of ranks $r_1$ and $r_2$ on ground sets $E_{M_1}$ and $E_{M_2}$ and let $A_i\subseteq E_{M_i}$ for $i=1,2$. Then we have the following equalities:
	\begin{align*}
		(M_1/A_1)\oplus (M_2/A_2) &= (M_1\oplus M_2)/(A_1\sqcup A_2) \\
		(M_1|A_1)\oplus (M_2|A_2) &= (M_1\oplus M_2)|(A_1\sqcup A_2)
	\end{align*}
\end{lem}
\begin{proof}
	We verify only the first equality above, as the other follows analogously. For $i=1,2$, suppose that $A_i$ satisfies $\rk_{\underline{M_i}}(A_i) = s_i$ and let $\mathbf{a}^{(i)}$ be a basis of the classical matroid $\underline{M_i}|A_i$. Then for $\mathbf{x}\in (E_{M_1}\setminus A_1)^{r_1-s_1}$ and $\mathbf{y}\in (E_{M_2}\setminus A_2)^{r_2-s_2}$ we see the following. 
	\begin{align*}
		((\mu_1/A_1)\oplus(\mu_2/A_2))(\mathbf{x},\mathbf{y}) &= \mu_1(\mathbf{x},\mathbf{a}^{(1)})\cdot\mu(\mathbf{y},\mathbf{a}^{(2)}) \\
		((\mu_1\oplus \mu_2)/(A_1\sqcup A_2))(\mathbf{x},\mathbf{y}) &= (\mu_1\oplus\mu_2)(\mathbf{x},\mathbf{y},\mathbf{a}^{(1)},\mathbf{a}^{(2)}) \\
		&= \mu_1(\mathbf{x},\mathbf{a}^{(1)})\cdot\mu_2(\mathbf{y},\mathbf{a}^{(2)}).
	\end{align*} 
	We thus see that the associated Grassmann-Pl\"{u}cker functions are equal and so must be the associated matroids. 
\end{proof} 

\begin{pro}\label{prop:pr_ex_dir_sum}
	The operation $\oplus$ is an exact direct sum on the category $F$-$\Mat$.
\end{pro}
\begin{proof}
	By the definition of the operation $\oplus$ for pointed $F$-matroids, the axiom (DS1) is immediate. To see that $\oplus:F\text{-}\Mat\times F\text{-}\Mat \to F\text{-}\Mat$ is exact, suppose that we have a short exact sequence
	\[\begin{tikzcd}
		(M_1,M_2)\ar[r,tail] & (X_1,X_2)\ar[r,two heads] & (N_1,N_2).
	\end{tikzcd}\]
	This means that $M_i\rightarrowtail X_i \twoheadrightarrow N_i$ is a short exact sequence for $i=1,2$, and so we may assume without loss of generality that there are subsets $A_i\subseteq E_{X_i}$ for which $M_i = X_i|A_i$ and $N_i = X_i/A_i$. We would like to show that the image of this sequence under $\oplus$ is a short exact sequence in $F\text{-}\Mat$. That is, we would like to see that $M_1\oplus M_2 \cong (X_1\oplus X_2)|(A_1\sqcup A_2)$ and $N_1\oplus N_2 \cong (X_1\oplus X_2)/(A_1\sqcup A_2)$. These isomorphisms in fact hold with equality by Lemma \ref{lem:ctr_del_dir_sum}. This establishes that (DS2) holds.
	
	To verify axiom (DS3), let $M,N,R$ be $F$-matroids on ground sets $E_M$, $E_N$, and $E_R$, respectively. Denote by $\iota_M:M\rightarrowtail M\oplus N$ and $\pi_M:M\oplus N \twoheadrightarrow M$ the associated inclusion and projection, and similar for $\iota_N,\iota_R,\pi_N,$ and $\pi_R$. We illustrate that the following map is an inclusion:
	\begin{align*}
		\text{Hom}(M\oplus N,R) &\longrightarrow \text{Hom}(M,R)\times\text{Hom}(N,R) \\
		f &\longmapsto (f\circ \iota_M,f\circ \iota_N)
	\end{align*}
	The second property regarding $\text{Hom}(R,M\oplus N)$ follows analogously and so we do not include the proof here. Suppose that $f,g:M\oplus N\to R$ are morphisms of $F$-matroids for which $f\circ \iota_M = g\circ \iota_M$ and $f\circ \iota_N = g\circ \iota_N$. Noting that $\iota_M$ is in fact the map $r_{E_M}:(M\oplus N)|E_M\rightarrowtail M\oplus N$, we note that it is represented by a block matrix $\begin{bmatrix} I & 0 \end{bmatrix}^T$, where $I$ is the $\tilde{E}_M\times\tilde{E}_M$ identity matrix and 0 is the $\tilde{E}_M\times\tilde{E}_N$ zero matrix. Likewise, $\iota_N=r_{E_N}$ can be represented by the block matrix $\begin{bmatrix} 0 & I \end{bmatrix}^T$ with blocks of the appropriate size. 
	
	The morphisms $f$ and $g$ are both represented by $\tilde{E}_R\times (\tilde{E}_M\sqcup\tilde{E}_N)$ matrices. Say $f$ has matrix $\begin{bmatrix} F_M & F_N \end{bmatrix}$ and $g$ has matrix $\begin{bmatrix} G_M & G_N \end{bmatrix}$, where $F_M$ and $G_M$ are of size $\tilde{E}_R\times\tilde{E}_M$ while $F_N$ and $G_N$ are $\tilde{E}_R\times \tilde{E}_N$. Then $f\circ \iota_M=g\circ\iota_M$ has matrix
	\[\begin{bmatrix} F_M & F_N \end{bmatrix}\begin{bmatrix} I \\ 0 \end{bmatrix} = \begin{bmatrix} F_M & 0 \end{bmatrix} = \begin{bmatrix} G_M & G_N \end{bmatrix}\begin{bmatrix} I \\ 0 \end{bmatrix}.\]
	Similarly, we conclude that $f\circ \iota_N=g\circ\iota_N$ has matrix 
	\[\begin{bmatrix} F_M & F_N \end{bmatrix}\begin{bmatrix} 0 \\ I \end{bmatrix} = \begin{bmatrix} 0 & F_N \end{bmatrix} = \begin{bmatrix} G_M & G_N \end{bmatrix}\begin{bmatrix} 0 \\ I \end{bmatrix}.\]
	Thus we see that $F_M=G_M$ and $F_N=G_N$, so that $f=g$. This shows that the above map of Hom-sets is an injection. Together with the analogous property for $\text{Hom}(R,M\oplus N)\to\text{Hom}(R,M)\times\text{Hom}(R,N)$, this establishes axiom (DS3). 
	
	For (DS4), suppose that $M\rightarrowtail X \twoheadrightarrow N$ is a short exact sequence in $F\text{-}\Mat$ with injection $\iota$ and projection $\pi$, for which we have a section $s$ of $\pi$ giving the following diagram in $F\text{-}\Mat$. 
	\[
	\begin{tikzcd} 
		& X & \\
		M \ar[ur, tail, "\iota"] \ar[r, tail, "\iota_M", swap] & M\oplus N \ar[u, "f", dashed]& N \ar[ul, "s", swap] \ar[l, tail, "\iota_N"]
	\end{tikzcd}  
	\] 
	We seek to find an isomorphism $f:M\oplus N\to X$ making the diagram above commute and to show that it is unique. As earlier, we assume without loss of generality that $M = X|A$ and $N = X/A$ for some $A\subseteq E_X$. Then $\pi = c_A$ is represented by the $(\tilde{E}_X\setminus A)\times \tilde{E}_X$ matrix $\begin{bmatrix} 0 & I \end{bmatrix}$, where $I$ is an $(\tilde{E}_X\setminus A)\times (\tilde{E}_X\setminus A)$ identity matrix. Similarly, $\iota = r_A$ has $\tilde{E}_X\times A$ matrix $\begin{bmatrix} I & 0 \end{bmatrix}^T$ with first block an $A\times A$ identity matrix. 
	
	The section $s:X/A\to X$ is represented by an $\tilde{E}_X\times (\tilde{E}_X\setminus A)$ matrix $\begin{bmatrix} 0 & S \end{bmatrix}^T$ where $S$ is an $(\tilde{E}_M\setminus A)\times (\tilde{E}_M\setminus A)$ matrix. This means that $\pi\circ s$ has matrix 
	\[\begin{bmatrix} 0 & I_{E_X\setminus A} \end{bmatrix}\begin{bmatrix} 0 \\ S \end{bmatrix} = \begin{bmatrix} S \end{bmatrix}.\]
	Since $\pi\circ s = \text{id}_N$, we see that $s$ has matrix $\begin{bmatrix} 0 & I_{E_X\setminus A} \end{bmatrix}^T$, hence $s=r_{E_X\setminus A}$ is an injection. We can now restate our diagram as follows:
	\[
	\begin{tikzcd} 
		& X & \\
		X|A \ar[ur, tail, "r_A"] \ar[r, tail, "r_A", swap] & (X|A)\oplus (X/A) \ar[u, "f", dashed]& X/A \ar[ul, "r_{E_X\setminus A}", swap] \ar[l, tail, "r_{E_X\setminus A}"]
	\end{tikzcd}  
	\] 
	Clearly one can choose $f$ to have the $\tilde{E}_X\times \tilde{E}_X$ identity matrix. This is in fact the only choice. To see this, suppose that $f$ is represented by the block matrix 
	\[\begin{bmatrix} F_{00} & F_{01} \\ F_{10} & F_{11} \end{bmatrix},\]
	where $F_{00}$ is $A\times A$ and $F_{11}$ is $(\tilde{E}_X\setminus A)\times (\tilde{E}_X\setminus A)$. Then the two equations 
	\[\begin{bmatrix} I_A \\ 0 \end{bmatrix} = \begin{bmatrix} F_{00} & F_{01} \\ F_{10} & F_{11} \end{bmatrix}\begin{bmatrix} I_A \\ 0 \end{bmatrix}\quad \text{and} \quad \begin{bmatrix} 0 \\ I_{E_X\setminus A} \end{bmatrix} = \begin{bmatrix} F_{00} & F_{01} \\ F_{10} & F_{11} \end{bmatrix}\begin{bmatrix} 0 \\ I_{\tilde{E}_X\setminus A} \end{bmatrix}\]
	imply that the only choice is $F_{00}=I_A$ and $F_{11} = I_{\tilde{E}_X\setminus A}$ with all other entries being zero. The argument regarding retractions follows similarly. 
\end{proof}

\begin{pro}\label{prop:comb_dir_sum} 
	$F\text{-}\Mat$ is a combinatorial proto-exact category with respect to the exact direct sum $\oplus$. 
\end{pro}
\begin{proof}
	Let $M_1,M_2$, and $N$ be $F$-matroids on ground sets $E_{M_1}$, $E_{M_2}$, and $E_N$. Given an  admissible mono $g:N\rightarrowtail M_1\oplus M_2$, we would like to exhibit admissible monos $\iota_k:N_k\rightarrowtail M_k$ for $k=1,2$ for which there is an isomorphism $f:N\to N_1\oplus N_2$ with $(\iota_1\oplus \iota_2)\circ f = g$. 
	
	As before, assume without loss of generality that $N = (M_1\oplus M_2)|(A_1\sqcup A_2)$ for some subsets $A_k\subseteq E_{M_k}$. We define $N_k:= N|A_k = (M_1\oplus M_2)|A_k = M_k|A_k$. Then by Lemma \ref{lem:ctr_del_dir_sum} we see
	\begin{align*}
		N_1\oplus N_2 &= (M_1|A_1)\oplus (M_2|A_2) \\
		&\cong (M_1\oplus M_2)|(A_1\sqcup A_2) = N
	\end{align*}
	and the equation $(\iota_1\oplus \iota_2)\circ f = g$ immediately follows.
\end{proof}

\begin{pro}\label{prop:pr_ex_duality}
	$F$-matroid duality $M\mapsto M^*$ endows $F$-$\Mat$ with the structure of a proto-exact category with duality.
\end{pro}
\begin{proof}
	Verification of the axioms (D1) and (D2) are immediate. Indeed, the zero element of $F\text{-}\Mat$ is the matroid $U_1^0$, whose self-duality establishes (D1). Further, (D2) follows from the fact that dual morphisms are given by transposing the corresponding matrices and the form of matrices representing admissible monos and epis. 
	
	To verify (D3), suppose that we have a biCartesian square in $F\text{-}\Mat$. Without loss of generality we may assume that it is of the following form for some matroid $M$ and disjoint subsets $A,B\subseteq E_M$.
	\begin{equation}
	\begin{tikzcd}
		M|A \ar[r, tail] \ar[d, two heads] & M \ar[d, two heads] \\
		(M|A)/B \ar[r, tail] & M/B
	\end{tikzcd}
	\end{equation}
	Recall that for any $S\subseteq E_M$ we have $(M|S)^*=M^*/S$ and $(M/S)^*=M^*|S$. Together with Proposition \ref{prop:restriction-contraction} we then see
	\[((M|A)/B)^* = (M|A)^*|B = (M^*/A)|B \cong (M^*|B)/A.\]
	So that applying duality to the above square yields
	\begin{equation}
	\begin{tikzcd}
		M^*/A & M^* \ar[l, two heads] \\
		(M^*|B)/A \ar[u, tail] & M^*|B \ar[u, tail] \ar[l, two heads]
	\end{tikzcd}
	\end{equation}
	and we know this to be biCartesian by Lemma \ref{lem:aux}.	
\end{proof}

\begin{pro}
	The category $F$-$\Mat$ is not split. In particular, for each perfect idyll $F$ there exist short exact sequences $M \rightarrowtail N \twoheadrightarrow P$ for which $N$ is not isomorphic to $M\oplus P$. 
\end{pro}
\begin{proof}
	Let $C_4$ be the cycle graph with edges $a,b,c,d$ and let $S = \{a\}$. Then, denoting by $M$ the (simple) graphic matroid of $C_4$, we have that $M$ is an $\mathbb{F}_1^{\pm}$-matroid (i.e. a regular matroid) and we have the following short exact sequence in $\mathbb{F}_1^\pm$-$\Mat$:
		\[M|S \rightarrowtail M \twoheadrightarrow M/S.\]
	Here $M|S$ is the graphic matroid of the path of length 1 and $M/S$ is the graphic matroid of the cycle $C_3$. Again, both matroids are simple. Since direct sums of graphic matroids correspond to disjoint union of the associated graphs, it is clear that $M$ is not isomorphic to $(M|S)\oplus (M/S)$. 
	
	For an arbitrary perfect idyll $F$, we can push the above example along the unique idyll morphism $\mathbb{F}_1^\pm\to F$ to realize this example in $\FMat$. 
\end{proof}

This establishes that $F$-$\Mat$ is neither split nor uniquely split for any perfect idyll $F$. 

\section{The category $F$-$\mathbf{SMat}_\bullet$}\label{sec: simple F matroids} 

Recall that an $F$-matroid $M$ is said to be simple precisely when its underlying matroid $\underline{M}$ is. The collection of all simple matroids with strong maps forms a reflective subcategory of the category of matroids with strong maps. Indeed, the assignment $M\mapsto \text{si}(M)$ is a functor from matroids to simple matroids, and this functor is left adjoint to the inclusion.

These statements also hold for pointed matroids by using the pointed simplification $\text{si}_\bullet$ of a pointed matroid $M$. In short, one forgets the distinguished loop of $M$, takes the non-pointed simplification, and takes the direct sum with $U^0_1$ to recover a pointed matroid. See \cite[Section 7]{heunen2017category} for more details on pointed and non-pointed simplification.

In our setting, the analogous statements are immediately verified: the collection of all pointed simple $F$-matroids together with $F$-matroid morphisms forms a reflective subcategory of $\FMat$. This follows because simplicity is defined purely in terms of the underlying matroid and because $\mathbb{K}$-$\Mat$ can be identified with the subcategory of $\Mat$ consisting of  $\mathbb{F}_1$-linear morphisms. 

Observe that all restrictions of simple matroids are necessarily again simple. In light of Proposition \ref{prop:ker and coker FMat} this observation establishes that $F$-$\textbf{SMat}_\bullet$ has all kernels. Since the simplification functor $\text{si}_\bullet$ is a left adjoint, we also see that it preserves colimits and so $F$-$\textbf{SMat}_\bullet$ also has cokernels.

\begin{rmk}[Kernels and cokernels]\label{rmk: kernels for simples} 
In fact, we can say more regarding kernels and cokernels in $F$-$\textbf{SMat}_\bullet$. Recall that the kernel of a morphism $f:M\to N$ in $\FMat$ is $M|\underline{f}^{-1}(*_N)$. If both $M$ and $N$ are simple, then by definition this means that $*_N$ is a flat of $\underline{N}$. Thus, by the fact that underlying map $\underline{f}:\underline{M}\to\underline{N}$ is strong, we see that $\underline{f}^{-1}(*_N)$ must be a flat of $\underline{M}$. This implies that any kernel in $F$-$\textbf{SMat}_\bullet$ is the restriction morphism associated to some flat $F\in\mathcal{L}(M)$. 

With regards to cokernels, observe that given a simple $F$-matroid $M$ and contraction morphism $c_A:M\to M/A$ in $\FMat$, applying $\text{si}_\bullet$ yields $\text{si}_\bullet(M)\to \text{si}_\bullet(M/A)$. A routine exercise in matroid theory yields that the underlying matroid of $\text{si}_\bullet(M/A)$ is $\underline{M}/\angles{A}$, where $\angles{A}\in\mathcal{L}(M)$ is the flat of $\underline{M}$ obtained as the closure of the subset $A$. Thus all contractions (and therefore all cokernels) in $F$-$\textbf{SMat}_\bullet$ are contractions at flats.
\end{rmk} 

As a reflective subcategory of $\FMat$, it is also immediate that $F$-$\textbf{SMat}_\bullet$ satisfies axiom (PA1) of proto-abelian categories (Subsection \ref{subsec: proto ab cat}).

\begin{pro}\label{prop:ker and coker FSMat}
	The category $F$-$\textbf{SMat}_\bullet$ admits all kernels and cokernels of morphisms. Moreover, it satisfies axiom (PA1) of a proto-abelian category.
\end{pro}

As a departure from the proto-exact structure on $\FMat$, we define the classes of admissible monomorphisms and epimorphisms yielding the proto-exact structure on $F$-$\textbf{SMat}_\bullet$ as follows. 

\begin{mydef}\label{definition: FSMat M and E}
	Let $f:M\to N$ be a morphism in $F$-$\textbf{SMat}_\bullet$. 	
	\begin{enumerate}
		\item 
		$f$ is an admissible monomorphism if $f=r_F \circ \alpha$, where $F \in\mathcal{L}(N)$ and $r_F:M|F \to M$ is an inclusion map, and $\alpha$ is an isomorphism of $F$-matroids: 
		\[\begin{tikzcd}
			M\ar[r,"\alpha", "\sim"'] & N|F\ar[r,"r_F"] & N,
		\end{tikzcd}\]
		Denote by $\MM_s$ the class of admissible monomorphisms in $F$-$\textbf{SMat}_\bullet$. 
		\item 
		$f$ is an admissible epimorphism if $f=\beta \circ c_G$, where $G \in\mathcal{L}(M)$ and $c_G:\mathcal{E} \to \mathcal{E}/G$ is a contraction map and $\beta$ is an isomorphism of $F$-matroids:
		\[\begin{tikzcd}
			M\ar[r,"c_G"] & M/G\ar[r,"\beta", "\sim"'] & N.
		\end{tikzcd}\]  
		Denote by $\EE_s$ the class of admissible epimorphisms in $F$-$\textbf{SMat}_\bullet.$ 
	\end{enumerate}
\end{mydef}

\begin{rmk}[Admissible morphisms]\label{rmk: FSMat M and E}
Observe that $\MM_s$ and $\EE_s$ are strictly smaller than the intersection of the classes $\MM$ and $\EE$ with the morphisms of the subcategory $F$-$\textbf{SMat}_\bullet$. For example, if $M$ is a simple matroid and there is some subset $A\subseteq E_M$ for which $M|A$ is simple but $A$ is not a flat of $M$, then the inclusion map $r_A:M|A\to M$ is in $\MM$ but is not in $\MM_s$. 

Despite this difference, however, note that $\MM$ and $\MM_s$ both consist of isomorphisms followed by kernels of morphisms in their respective categories. Likewise, $\EE$ and $\EE_s$ consist of cokernels of morphisms followed by isomorphisms. A similar characterization of admissibles will be used for the proto-exact structure on tropical reflexive sheaves in Section \ref{sec: trs proto ex}.  
\end{rmk} 

Observe that Propositions \ref{prop:contian iso} and \ref{prop:composition}, as well as Lemma \ref{lem:aux}, hold in $F$-$\textbf{SMat}_\bullet$. The latter then gives that Propositions  \ref{prop:LR bicart}, and \ref{prop:cart iff cocart} hold as well. Lemma \ref{lem:aux} is true on the nose with arrows coming from $\MM_s$ and $\EE_s$ (so that $A$ and $B$ should be flats), while the other propositions follow essentially from the observation that if $N$, $N|A$, and $N/B$ are simple matroids, so must be $(N|A)/B$. Together with Proposition \ref{prop:protoex_then_protoab} and Proposition \ref{prop:ker and coker FSMat}, we obtain the following corollary. 

\begin{cor}\label{cor: FSMat proto ab}
	For each perfect idyll $F$, the category $F$-$\textbf{SMat}_\bullet$ is proto-abelian. 
\end{cor} 

Sums and intersections of subobjects in $F$-$\textbf{SMat}_\bullet$ depart slightly from the characterization of those in $\FMat$ as given in Proposition \ref{prop: sum prod of subobj FMat}. Instead, we have the following:

\begin{pro}\label{prop: sum prod of subobj FSMat}
	Let $M_1,M_2\rightarrowtail M$ be strict subobjects of $M$ in $F$-$\textbf{SMat}_\bullet$ given as restrictions $M_1= M|F_1$ and $M_2= M|F_2$ at flats $F_1,F_2\in\mathcal{L}(M)$. Then the sum and intersection of the subobjects $M_1$ and $M_2$ are given by
	\begin{align*}
		(M|F_1)+(M|F_2) &= M|(F_1 \vee F_2) \\
		(M|F_1)\cap(M|F_2) &= M|(F_1 \wedge F_2).
	\end{align*}
\end{pro}

The characterization of the intersection is identical to $\FMat$, as the intersection of $F_1$ and $F_2$ is precisely $F_1\wedge F_2$. The characterization of the sum follows by reasoning identical to that regarding contractions of simple matroids. In particular, when contracting $F_1\cup F_2$ one must instead contract $\angles{F_1\cup F_2} = F_1\vee F_2$ in order to obtain a simple $F$-matroid.

\begin{rmk}[No duality] \label{rmk: FSMat not simple}
Since the dual of a simple matroid is not necessarily simple, we see that $F$-$\textbf{SMat}_\bullet$ is not proto-exact with duality as $\FMat$ is. Indeed, one easily finds simple planar graphs whose planar duals are not simple. Taking the associated graphic matroids then yields two dual $\mathbb{F}_1$-matroids (i.e. regular matroids), one of which is simple and the other of which is not. Since $\mathbb{F}_1$ is initial in the category of idylls, this counterexample can be realized in all categories $F$-$\textbf{SMat}_\bullet$. 
\end{rmk}

\begin{rmk}[Artinian and Noetherian conditions]\label{rmk: FSMat art noeth}
Observe that $F$-$\textbf{SMat}_\bullet$ inherits the Artinian and Noetherian properties of $\FMat$ described by Proposition \ref{pro: art noeth FMat}. This is an immediate consequence of the fact that $F$-$\textbf{SMat}_\bullet$ is a subcategory of $\FMat$. 	
\end{rmk}

\section{Hall algebras and $K$-theory of $F$-matroids}\label{section: Hall algebra and K-theory} 

\begin{pro}
	For a finite perfect idyll $F$, the category $F$-$\Mat$ is finitary. In particular, $\mathbb{K}$-$\Mat$, $\mathbb{S}$-$\Mat$, and $\mathbb{F}_1^\pm$-$\Mat$ are finitary categories.
\end{pro}
\begin{proof}
	We show this directly, by illustrating that for any two $F$-matroids $M$ and $N$ that the sets $\text{Hom}(M,N)$ and $\text{Ext}^1(M,N)$ are finite. 
	Let $M$ and $N$ be $F$-matroids on ground sets $E_M$ and $E_N$, and let $f:M\to N$ be a morphism. Then $f$ is given by a submonomial matrix with entries in $F$, hence $\text{Hom}(M,N)$ is necessarily finite by the finiteness of $F$. 
	
	Suppose now that $P$ is an extension of $N$ by $M$ of rank $r_P$. Then that we have a short exact sequence
	\[
	M \rightarrowtail P \twoheadrightarrow N.
	\]
	By the exactness of the above sequence, we see that $P$ is necessarily a matroid on a ground set $E_P$ of cardinality $|\tilde{E}_M\sqcup \tilde{E}_N\sqcup\{*_P\}|$, so that we can obtain the desired finiteness by illustrating that there are only finitely many such $F$-matroids. But this is necessarily the case since the number of such matroids $P = [\pi]$ is bounded by the number of Grassmann-Pl\"{u}cker maps $\pi:E_P^{r_P}\to F$. The common domain and codomain of these maps is finite, hence there are only finitely many such maps and so finitely many such $F$-matroids $P$. 	 
\end{proof}

The above proposition now establishes that $F\text{-}\Mat$ is finitary for $F=\mathbb{K}, \mathbb{S}$, and $\mathbb{F}_1^\pm$. However for $F=\mathbb{T}$ the situation is slightly more delicate. In particular, $\mathbb{T}$ is not finite, and so the above argument does not apply.  

From Theorem \ref{thm:F mat proto ex}, one obtains the Hall algebra $H_{\FMat}$ for the category $\FMat$.

\begin{myeg}
	One can easily check that $H_{\mathbb{K}\text{-}\Mat}$ is the Hopf dual of matroid-minor Hopf algebra. In fact, in \cite{eppolito2018proto}, it is shown that the matroid-minor Hopf algebra is the Hopf dual of the Hall algebra $H_{\Mat}$. 
\end{myeg}

\subsection{Hall algebra of $F$-$\Mat$ and the $F$-matroid minor Hopf algebra} 

Let $F$ be a perfect idyll, and $\mathcal{M}$ be a set of $F$-matroids which is closed under taking minors and direct sums. Let $\FMat(\mathcal{M})$ be the subcategory of $\FMat$ whose objects are $F$-matroids in $\mathcal{M}$. As in \cite[Section 7.2]{eppolito2018proto} (or \cite[Theorem 5.11]{eppolito2018proto}), one can easily see that $\FMat(\mathcal{M})$ is a full subcategory of $\FMat$ and the proto-exact structure of $\FMat$ descends to $\FMat(\mathcal{M})$. In particular, if $F$ is finite then we can obtain the Hall algebra $H_{\FMat(\mathcal{M})}$ associated to  $\FMat(\mathcal{M})$.

In \cite{eppolito2020hopf}, the classical construction of matroid-minor Hopf algebras was generalized to the case of matroids over hyperfields, and hence to the case of matroids over idylls. Let $\mathcal{M}_{\text{iso}}$ be the set of isomorphism classes of $F$-matroids in $\mathcal{M}$ and $k[\mathcal{M}_{\text{iso}}]$ be the associated matroid-minor Hopf algebra as in \cite{eppolito2020hopf}, where $k$ is a field of characteristic zero. We can generalize \cite[Theorem 7.3]{eppolito2018proto} as follows. 

\begin{cor}
	Let $F$ be a finite perfect idyll. Let $\mathcal{M}$ be a collection of $F$-matroids closed under taking minors and direct sums. Then, we have
	\[
	H_{\FMat(\mathcal{M})} \cong k[\mathcal{M}_{\text{iso}}]^*,
	\]
	where $k[\mathcal{M}_{\text{iso}}]^*$ is the graded Hopf dual of $k[\mathcal{M}_{\text{iso}}]$. Furthermore, we have
	\[
	H_{\FMat(\mathcal{M})} \cong \mathbf{U}(\delta_{[M]}), \quad [M] \in \mathcal{M}^{\text{ind}}_{\text{iso}},
	\]
	where $\mathcal{M}^{\text{ind}}_{\text{iso}}$ denotes the isomorphism classes of connected $F$-matroids in $\mathcal{M}$. 
\end{cor}
\begin{proof} 
	We first note that the second assertion directly follows from \cite[Theorem 7.1]{eppolito2018proto}.
	
	For the first assertion, note that the notion of minors and direct sums for matroids over hyperfields in \cite[Section 3]{eppolito2020hopf} agrees with the notion of minors and direct sums for matroids over idylls in \cite[Section 1]{jarra2024quiver}. 
	
	Now, with the identification of $H_{\FMat(\mathcal{M})}$ and $k[\mathcal{M}_{\text{iso}}]^*$ with the following
	\[
	\{f:\mathcal{M}_{\text{iso}} \to k \mid |\text{supp}(f)|<\infty\}, 
	\]
	one can easily see that the unit, co-unit, and coproduct agree. For multiplication, in $H_{\FMat(\mathcal{M})}$, one has
	\[
	(f\cdot g)([M])=\sum_{N \subseteq M}f([M/N])g([N]). 
	\]
	On the other hand, in $k[\mathcal{M}_{\text{iso}}]^*$, one has 
	\[
	(f\cdot g)([M])=\sum_{N \subseteq M}f([N])g([M/N]). 
	\]
	But, every enveloping algebra possesses an algebra anti-automorphism fixing the unit, co-unit, and coproduct. Hence, the first assertion follows from the second assertion. 
\end{proof} 

\subsection{$K$-Theory of $F$-$\Mat$}

Fix a perfect idyll $F$. In this subsection we compute the Grothendieck group $K_0(F$-$\Mat)$ and show that it coincides with the Grothendieck group $K_0(\mathbf{Mat}_\bullet)$ of pointed (classical) matroids, as computed in \cite{eppolito2018proto}. This illustrates in particular that the Grothendieck group is not sensitive to the coefficients used to define the associated category of matroids. 

\begin{pro}\label{proposition: $K_0$ computation}
	We have $K_0(F$-$\Mat)\cong \mathbb{Z}\oplus\mathbb{Z}$.
\end{pro}
\begin{proof}
	Given $M=[\mu]\in F$-$\Mat$, fix some non-loop element $e\in E_M$. Then we have the following short exact sequence:
	\[
	M|e \rightarrowtail M \twoheadrightarrow M/e.
	\]
	Thus in $K_0(F$-$\Mat)$ we have the equation $[M] = [M|e] + [M/e]$. Note that $M|e$ is now an $F$-matroid of rank 1 on a ground set consisting of a non-loop and the basepoint. Evaluating the associated Grassmann-Pl\"{u}cker function, we then see
	\[(\mu|e)(e') = \begin{cases} \alpha & e'=e \\ 0_F & e'=*_M \end{cases}\]
	for some $\alpha\in F^\times$, so that in particular we have an isomorphism $M|e \cong C$, where $C$ is the matroid consisting of one non-loop and the basepoint. We would like to perform this process $\rk(\underline{M})$ times. To this end, let $B\subseteq \tilde{E}_M$ be a basis of $\underline{M}$ and perform the above process for each $e\in B$. This yields the following decomposition: 
	\[[M] = \rk(\underline{M})[C] + [M/B].\]
	On the other hand, when we contract all elements of $B$, we see that the resulting matroid consists only of loops and the basepoint $*_M$. However, appealing to the definition of the direct sum on $F$-$\Mat$, we then see that 
	\[M/B\cong L \oplus \cdots \oplus L,\]
	where $L$ is the matroid with two loops, one distinguished, and the number of summands is equal to the corank (or nullity) of $M$. Denoting this number by $\ns(M)$, we then see that
	\[[M] = \rk(\underline{M})[C] + \ns(M)(\underline{M})[L].\] 
	Since an $F$-matroid can have arbitrary rank and an arbitrary number of loops, this establishes the desired result.
\end{proof}

\begin{pro}
	Let $\Phi:F \to G$ be a morphism of idylls. Then, for each $n \in \mathbb{N}$, $\Phi$ induces a morphism 
	\[
	\Phi_*:K_n(\FMat) \to K_n(\text{G-}\Mat).
	\]
\end{pro}
\begin{proof}
	By Remark \ref{rmk:base_change_functor} we have that $\Phi$ induces a functor from $\FMat$ to $G$-$\Mat$. By the definition of restrictions and contractions we then immediately see that this functor preserves short exact sequences, hence induces a morphism on $K_n$ for each $n$.
\end{proof}

\begin{rmk}
	Proposition \ref{proposition: $K_0$ computation} suggests that to see any difference between $K$-theory of $\FMat$, one may have to go higher $K$-theory. In fact, from the adjunction between $\Mat$ and $\textbf{FinSet}_{\bullet}$, one obtains the following:
	\[
	\pi_n^s(\mathbf{S}) \cong K_n(\textbf{FinSet}_{\bullet}) \hookrightarrow K_n(\Mat) \text{ for all } n \in \mathbb{N},
	\] 
	where $\pi_n^s(\mathbf{S})$ denotes the stable homotopy groups of the sphere spectrum (see \cite{eppolito2018proto} for further details). So, it would be interesting to compute even $K_1(\Mat)$ to see whether or not the above inclusion is proper.	
\end{rmk} 

\section{The category $\TRS$}\label{sec: trs proto ex}

We now examine the categorical structure of tropical toric reflexive sheaves. In this section, all tropical toric reflexive sheaves, $\mathbb{T}$-matroids, and morphisms thereof are assumed to be pointed. Let $L$ be a free abelian group of finite rank and fix a complete rational fan $\Sigma$ in a vector space $L_\mathbb{R}$ and let $\Lambda = \Hom(L,\mathbb{Z})$ be the dual lattice. Denote by $X_\Sigma$ the associated complete toric variety and $\text{trop}(X_\Sigma)$ the tropical toric varierty with fan $\Sigma$ as in \cite{kajiwara2008tropical} or \cite{payne2009analytification}.

\begin{mydef}\label{definition: morphism of trs}
	Let $\mathcal{E}=(M,\{F^\rho_\bullet\}_\rho)$ and $\mathcal{F}=(N,\{G^\rho_\bullet\}_\rho)$ be tropical toric reflexive sheaves on $\text{trop}(X_\Sigma)$ (Section \ref{subsec: ttvb}). A \emph{morphism} $f:\mathcal{E}\to\mathcal{F}$ is a pair $f=(f_\mathbb{T},u)$, where $f_\mathbb{T}$ is a morphism of simple $\mathbb{T}$-matroids $f_\mathbb{T}:M\to N$ and $u\in \Lambda$ is a vector such that for each ray $\rho\in\Sigma(1)$ and each $j\in\mathbb{Z}$, we have 
	\begin{equation}\label{eq: toric ref sheav}
		f(F^\rho_j) \subseteq G^\rho_{j+u\cdot v_\rho} 
	\end{equation}
	Recall that $v_\rho\in L$ is the primitive lattice vector associated to the ray $\rho$.
\end{mydef}

It is clear that the tropical toric reflexive sheaves form a category, and for morphisms $f = (f_\mathbb{T},u):\mathcal{E}\to\mathcal{F}$ and $g = (g_\mathbb{T},w):\mathcal{F}\to\mathcal{G}$ the composition is given by $g\circ f = (g_\mathbb{T}\circ f_\mathbb{T}, u+w)$. We denote the resulting category by $\TRS$. Note that the definition of morphisms in $\TRS$ yields a forgetful functor $\TRS\to\mathbb{T}\text{-}\Mat$ which simply forgets the filtrations associated to objects and the translations $u$ associated to morphisms. 

Under this definition of morphism, observe that two tropical toric reflexive sheaves are isomorphic precisely when they are isomorphic in the sense of \cite{khan2024tropical}. 

\begin{lem}\label{lem: trs isos}
	Let $\mathcal{E}=(M,\{F^\rho_\bullet\}_\rho)$ and $\mathcal{F}=(N,\{G^\rho_\bullet\}_\rho)$ be tropical toric reflexive sheaves on $\text{trop}(X_\Sigma)$
	Suppose that $f=(f_\mathbb{T},u):\mathcal{E}\to\mathcal{F}$ and $g=(f_\mathbb{T}^{-1},w):\mathcal{F}\to\mathcal{E}$ are morphisms in $\TRS$ for which $f_\mathbb{T}^{-1}=(f_\mathbb{T})^{-1}$ in $\mathbb{T}$-$\Mat$. Then $w=-u$.
\end{lem}
\begin{proof}
	Note that for all $\rho\in\Sigma(1)$ and all $j\in\mathbb{Z}$ we have
		\begin{align*}
			f_\mathbb{T}(F^\rho_j) &\subseteq G^\rho_{j+u\cdot v_\rho} \\
			f^{-1}_\mathbb{T}(G^\rho_j) &\subseteq F^\rho_{j+w\cdot v_\rho}.
		\end{align*}
	Then we may write
		\begin{align*}
			F^\rho_j = f_\mathbb{T}^{-1}(f_\mathbb{T}(F^\rho_j))&\subseteq f^{-1}_\mathbb{T}(G_{j+u\cdot v_\rho}) \subseteq F^\rho_{j+(u+w)\cdot v_\rho} \\
			G^\rho_j = f_\mathbb{T}(f^{-1}_\mathbb{T}(G^\rho_j)) &\subseteq G^\rho_{j+(u+w)\cdot v_\rho}.
		\end{align*}
	Now suppose for the sake of contradiction that $u+w\neq 0$. Then by the assumption that $\Sigma$ is a complete fan, there is some ray $\rho_+\in\Sigma(1)$ with primitive generator $v_+$ for which $(u+w)\cdot v_+ > 0$. Then since $F^{\rho_+}_\bullet$ is a decreasing chain of flats we see that $F^{\rho_+}_{j+(u+w)\cdot v_+}\subseteq F^{\rho_+}_j$ as well, hence $F^{\rho_+}_j = F^{\rho_+}_{j+(u+w)\cdot v_+}$ for each $j\in\mathbb{Z}$. But then the chain $F^{\rho_+}_\bullet$ is constant, violating the definition of a tropical toric reflexive sheaf. Thus it cannot be that $u+w\neq 0$.  
\end{proof}

\begin{pro}\label{pro: trs isos}
	Two tropical reflexive sheaves $\mathcal{E}=(M,\{F^\rho_\bullet\}_\rho)$ and $\mathcal{F}=(N,\{G^\rho_\bullet\}_\rho)$ are isomorphic in the sense of Definition \ref{def: iso of trs} if and only if there is an isomorphism $f:\mathcal{E}\to \mathcal{F}$ in $\TRS$. 
\end{pro}
\begin{proof}
	If $\mathcal{E}$ and $\mathcal{F}$ are isomorphic in the sense of Definition \ref{def: iso of trs}, then there is a $\mathbb{T}$-matroid isomorphism $f_\mathbb{T}:M\to N$ and a vector $u\in \Lambda$ for which
		\[f_\mathbb{T}(F^\rho_j) = G^\rho_{j+u\cdot v_\rho}\]
	for every $\rho\in\Sigma(1)$ and every $j\in\mathbb{Z}$. Denote by $(f_\mathbb{T})^{-1}$ the inverse of $f_\mathbb{T}$ in $\mathbb{T}$-$\Mat$. Then by definition $f:=(f_\mathbb{T},u)$ and $f^{-1}:=((f_\mathbb{T})^{-1},-u)$ are morphisms in $\TRS$ which are mutually inverse. 
	
	Conversely, let $f=(f_\mathbb{T},u)$ be an isomorphism from $\mathcal{E}$ to $\mathcal{F}$ in $\TRS$. Then by definition there is an inverse morphism in $\TRS$ which by Lemma \ref{lem: trs isos} we can write as $f^{-1} = (f_\mathbb{T}^{-1},-u)$. That is, we have
		\begin{align*}
			f_\mathbb{T}(F_j^\rho)&\subseteq G^\rho_{j+u\cdot v_\rho} \\
			f_\mathbb{T}^{-1}(G^\rho_{j+u\cdot v_\rho}) &\subseteq F^\rho_{j+(-u+u)\cdot v_\rho} = F^\rho_j. 
		\end{align*} 
	Thus by applying $f_\mathbb{T}$ to both sides of the second equality we see that in fact $f_\mathbb{T}(F^\rho_j) = G^\rho_{j+u\cdot v_\rho}$, establishing that $\mathcal{E}$ and $\mathcal{F}$ are isomorphic in the sense of Definition \ref{def: iso of trs}.
\end{proof}

This illustrates that isomorphism in $\TRS$ and isomorphism in the sense of Definition \ref{def: iso of trs} are the same.  

\begin{mydef}\cite[Definitions 7.2 and 7.17]{khan2024tropical}\label{definition: rest and contrac}
	Let $\mathcal{E}=(M,\{E^\rho_\bullet\}_\rho)$ be a tropical toric reflexive sheaf and take $F \in \mathcal{L}(M)$. 
	\begin{enumerate}
		\item 
		The restriction of $\mathcal{E}$ to $F$ is the tropical reflexive sheaf $\mathcal{E}|F = (M|F, \{F^\rho_\bullet\}_\rho)$, where 
		\[
		F^\rho_j:=F \wedge E^\rho_j \in \mathcal{L}(M).
		\]
		\item 
		The contraction (or quotient) of $\mathcal{E}$ by $F$ is the tropical reflexive sheaf $\mathcal{E}/F = (M/F, \{G^\rho_\bullet\}_\rho)$, where $G^\rho_j\in\mathcal{L}(M/F)$ is the flat such that under the lattice isomorphism $\varphi:\mathcal{L}(M/F)\to [F,E_M]$ we have
		\[
		\varphi(G^\rho_j)=F\vee E^\rho_j \in \mathcal{L}(M).
		\]
	\end{enumerate}
\end{mydef}

\begin{mydef}\label{definition: admissible}
	With the same notation as above, for $A \subseteq E_M$, we let $\mathcal{E}|A$ (resp.~$\mathcal{E}/A$) be $\mathcal{E}|{\angles{A}}$ (resp.~$\mathcal{E}/{\angles{A}}$), where $\angles{A}$ is the flat generated by $A$.\footnote{This definition makes sense since the intersection of two flats is again a flat.} 
\end{mydef}

\begin{lem}\label{lemma: morphism}
	Let $\mathcal{E}=(M,\{E^\rho_\bullet\}_\rho)$ be a tropical toric reflexive sheaf. Let $F \in \mathcal{L}(M)$. Then the following hold:
	\begin{enumerate}
		\item 
		The restriction map $r_F:M|F\to M$ induces a morphism $\iota = (r_F,0):\mathcal{E}|F \to \mathcal{E}$ of tropical toric reflexive sheaves.
		\item
		The contraction map $c_F:M\to M/F$ induces a morphism $\pi = (c_F, 0):\mathcal{E} \to \mathcal{E}/F$ of tropical toric reflexive sheaves. 
	\end{enumerate} 
\end{lem}
\begin{proof} 
	(1): Let $\mathcal{E}|F=(M|F,\{F^\rho_\bullet\})$. Note that for $F_1,F_2 \in \mathcal{L}(M)$, we have $F_1 \cap F_2 = F_1 \wedge F_2$. It follows that
	\[
	\underline{r_F}(E^\rho_j)=F^\rho_j=F\wedge E^\rho_j \subseteq E^\rho_j,
	\]	
	showing that $\iota=(r_F,0)$ is a morphism.
	
	(2): Recall that for $F \subseteq E_M$, with $\widetilde{F}:=F\backslash \{*_M\}$, one has
	\[
	\mathcal{L}(M/F)=\{A \backslash \widetilde{F} \mid \widetilde{F} \subseteq A \in \mathcal{L}(M)\} \cong \{A\vee F \mid A\in\mathcal{L}(M)\}.
	\]
	Let $\mathcal{E}/F=(M/F,\{G^\rho_\bullet\})$. From Definition \ref{definition: rest and contrac}, we have that $G^\rho_j$ can be identified with $E^\rho_j\vee F$. We thus see 
	\[
	\underline{c_F}(E^\rho_j) = G^\rho_j, 
	\]	
	showing that $\pi=(c_F,0)$ is a morphism of tropical toric reflexive sheaves.  
\end{proof} 

As in the category $\mathbb{T}$-$\textbf{SMat}_\bullet$, restrictions and contractions of tropical reflexive sheaves at flats furnish $\TRS$ with kernels and cokernels. 

\begin{pro}\label{pro: ker coker in trs}
	Let $\mathcal{E} = (M, \{E^\rho_\bullet\}_\rho)$ and $\mathcal{F} = (N, \{F^\rho_\bullet\}_\rho)$ be tropical reflexive sheaves in $\TRS$ and let $f:\mathcal{E}\to\mathcal{F}$ be a morphism. Then the inclusion $\mathcal{E}|\underline{f_\mathbb{T}}^{-1}(*_N)\to \mathcal{E}$ is a kernel of $f$ and the contraction $\mathcal{F}\to \mathcal{F}/\angles{\underline{f_\mathbb{T}}(E_M)}$ is a cokernel of $f$.
\end{pro}
\begin{proof}
	Applying the forgetful functor to $f:\mathcal{E}\to\mathcal{F}$ we obtain the map $f_\mathbb{T}:M\to N$ in $\mathbb{T}$-$\textbf{SMat}_\bullet$, where we find a kernel $M|\underline{f_\mathbb{T}}^{-1}(*_N)$ of $f_\mathbb{T}$. By virtue of Definition \ref{definition: rest and contrac}, the restriction has a natural tropical reflexive sheaf structure which is immediately verified to satisfy the universal property of kernels. The proof for cokernels is similar and has been omitted.
\end{proof}

Appealing to Proposition \ref{prop:ker and coker FSMat} with $F=\mathbb{T}$, we see that $\TRS$ also satisfies axiom (PA1) from Subsection \ref{subsec: proto ab cat}.

\begin{pro}\label{pro: ker coker pa1 trs}
	$\TRS$ has all kernels and cokernels and satisfies axiom (PA1) of a proto-abelian category.
\end{pro}

We now define admissible monomorphisms $\mathfrak{M}_t$ and admissible epimorphisms $\mathfrak{E}_t$, which we later see furnish $\TRS$ with a proto-exact structure. 

\begin{mydef}\label{definition: TRS M and E}
	Let $f:\mathcal{E}=(M,\{E^\rho_\bullet\}_\rho) \to \mathcal{F}=(N,\{F^\rho_\bullet\}_\rho)$ be a morphism in $\TRS$. 	
	\begin{enumerate}
		\item 
		$f$ is an admissible monomorphism if $f=r_F \circ \alpha$, where $F\in\mathcal{L}(N)$ and $r_F:\mathcal{F}|F \to \mathcal{F}$ is an inclusion map, and $\alpha$ is an isomorphism of tropical toric reflexive sheaves. Denote by $\MM_t$ the class of admissible monomorphisms. 
		\item 
		$f$ is an admissible epimorphism if $f=\beta \circ c_G$, where $G \in\mathcal{L}(M)$ and $c_G:\mathcal{E} \to \mathcal{E}/G$ is a contraction map and $\beta$ is an isomorphism of tropical toric reflexive sheaves. Denote by $\EE_t$ the class of admissible epimorphisms. 
	\end{enumerate}
\end{mydef}

\begin{rmk}\label{rmk: TRS M and E}
	As in $\FMat$ and $F$-$\textbf{SMat}_\bullet$, the classes of admissible monomorphisms and admissible epimorphisms consist of kernels and cokernels of $\TRS$ composed with isomorphisms. However, note that as with $F$-$\textbf{SMat}_\bullet$ these classes are only defined in terms of the flats of the underlying matroid, rather than with arbitrary subsets.
\end{rmk} 

The following are straightforward. 

\begin{lem}
	The following hold for $(\TRS,\mathfrak{M}_t,\mathfrak{E}_t)$.
	\begin{enumerate}
		\item 
		$\TRS$ is pointed.
		\item 
		The classes $\mathfrak{M}_t$ and $\mathfrak{E}_t$ contain all isomorphisms. 
		\item 
		The classes $\mathfrak{M}_t$ and $\mathfrak{E}_t$ are closed under composition. 
	\end{enumerate}
\end{lem}
\begin{proof}
	(1): One readily check that the zero object is the tropical toric reflexive sheaf whose $\mathbb{T}$-matroid is $U^0_1$ with the trivial filtration. 
	
	(2): Let $f:\mathcal{E}=(M,\{E^\rho_\bullet\}) \to \mathcal{F}=(N,\{F^\rho_\bullet\})$ be an isomorphism. For $\mathfrak{M}_t$, one can easily see that $f=r_{E_N}\circ f$. For $\mathfrak{E}_t$, likewise, one can see that $f=f\circ c_{\emptyset}$. 
	
	(3): This directly follows from Proposition \ref{prop:composition} and the definition of compositions in $\TRS$.
\end{proof}

\begin{lem}\label{lemma: TRS bi-Cartesian lemma}
	Let $\mathcal{E}=(M,\{E^\rho_\bullet\})$ be a tropical toric reflexive sheaf and let $G\leq F\in\mathcal{L}(M)$. The following square is biCartesian in $\TRS$:
	\begin{equation}\label{eq: diagram}
		\begin{tikzcd}
			\mathcal{E}|F \ar[r, tail, "f"] \ar[d, two heads, "h"] & \mathcal{E} \ar[d, two heads, "k"] \\
			(\mathcal{E}|F)/G \ar[r, tail, "g"] & \mathcal{E}/G
		\end{tikzcd}
	\end{equation}
	where $f = r_F$ and $g = r_{(F\setminus G)}$.
\end{lem}
\begin{proof}
	We first prove that \eqref{eq: diagram} is co-Cartesian. Suppose that we have the following diagram in $\TRS$ with $\mc{F} = (N,\{F^\rho_\bullet\}_\rho)$.
	\begin{equation}\label{eq: cocart}
		\begin{tikzcd}
			\mathcal{E}|F \ar[r,"f"] \ar[d,"h"] & \mathcal{E} \ar[d,"k"] \ar[ddr,"\beta"] & \\
			(\mathcal{E}|F)/G \ar[r,"g"] \ar[drr,"\alpha",swap] & \mathcal{E}/G & \\
			& & \mathcal{F}
		\end{tikzcd}
	\end{equation}
	Since all morphisms in \eqref{eq: cocart} are morphisms in $\mathbb{T}$-$\textbf{SMat}_\bullet$, from Lemma \ref{lem:aux}, we have a unique morphism $\gamma_\mathbb{T}:M/G \to N$ which makes \eqref{eq: cocart} commute at the level of $\mathbb{T}$-matroids. We only have to check that $\gamma_\mathbb{T}$ defines a morphism in $\TRS$. In other words, we need to show that there is some $w\in\Lambda$ with
\[
\underline{\gamma_\mathbb{T}}(E^\rho_j \vee G) \subseteq F^\rho_{j+w\cdot v_\rho}
\]
But, since $\beta=(\beta_\mathbb{T},u)$ is a morphism of tropical toric reflexive sheaves, we have 
\[
\underline{\gamma_\mathbb{T}}(E^\rho_j\vee G) = \underline{\gamma_\mathbb{T}}(k_\mathbb{T}(E^\rho_j)) = \underline{\beta_\mathbb{T}}(E^\rho_j) \subseteq F^\rho_{j+u\cdot v_\rho},
\]	
showing that $\gamma=(\gamma_{\mathbb{T}},u)$ is a morphism in $\TRS$. 
	
The case for Cartesian is similar, but we just include it here for the completeness. Suppose that we have the following diagram in $\TRS$ with $\mc{G} = (R, \{G^\rho_\bullet\}_\rho)$. 
	\begin{equation}\label{eq: cart}
		\begin{tikzcd}
			\mathcal{G} \ar[ddr,"\alpha"] \ar[drr,"\beta"] & &  \\ 
			& \mathcal{E}|F \ar[r,"f",swap] \ar[d,"h"] & \mathcal{E} \ar[d,"k"]  \\ 
			& (\mathcal{E}|F)/G \ar[r,"g"] & \mathcal{E}/G \\ 
		\end{tikzcd}
	\end{equation}
	Again, from Lemma \ref{lem:aux}, we have a unique morphism $\gamma_\mathbb{T}:R \to M|F$ which makes \eqref{eq: cart} commute at the level of $\mathbb{T}$-matroids. We only have to check that $\gamma_{\mathbb{T}}$ defines a morphism in $\TRS$. We need to show that there is some $w\in\Lambda$ with
\[
\gamma(G^\rho_j) \subseteq F^\rho_{j+w\cdot v_\rho} = E^\rho_{j+w\cdot v_\rho} \wedge F.
\]
But, since $r_F\circ \gamma_\mathbb{T} = \beta_\mathbb{T}$ and $\beta=(\beta_\mathbb{T},u)$ is a morphism in $\TRS$, we have
\[
\underline{f_\mathbb{T}}(\underline{\gamma_\mathbb{T}}(G^\rho_j)) = \underline{\beta_\mathbb{T}}(G^\rho_j) \subseteq E^\rho_{j+u\cdot v_\rho}.
\]	
It follows that
\[
\underline{\gamma_\mathbb{T}}(G^\rho_j) \subseteq \underline{f_\mathbb{T}}^{-1}(E^\rho_{j+u\cdot v_\rho})=E^\rho_{j+u\cdot v_\rho}\wedge F=F^\rho_{j+u\cdot v_\rho},
\]	
showing that $\gamma=(\gamma_\mathbb{T},u)$ is a morphism in $\TRS$.	 
\end{proof}

\begin{pro}\label{proposition: completing square}
	Any diagram in $\TRS$ of the form $\mathcal{E}\rightarrowtail \mathcal{G} \twoheadleftarrow \mathcal{F}$ can be completed to a biCartesian square.
	
\end{pro}
\begin{proof}
	The same argument as in the proof of Proposition \ref{prop:LR bicart} can be used to prove this by using Lemma \ref{lemma: TRS bi-Cartesian lemma}. 
\end{proof}

\begin{pro}\label{proposition: completing the square}
	Any diagram in $\TRS$ of the form $\mathcal{E}\twoheadleftarrow \mathcal{G} \rightarrowtail \mathcal{F}$ can be completed to a biCartesian square. 
\end{pro}
\begin{proof}
	As with $F$-$\textbf{SMat}_\bullet$, we cannot appeal to duality here as we did in $\FMat$, but as described after Proposition \ref{prop:UL bicart} one readily follows steps in analogy with the proof of Proposition \ref{prop:LR bicart}. 
\end{proof}

\begin{pro}
	Any square of the following form in $\TRS$ is Cartesian if and only if it is coCartesian:
	\[
	\begin{tikzcd}
		\mathcal{E} \ar[r, tail] \ar[d, two heads] & \mathcal{F} \ar[d, two heads] \\
		\mathcal{G} \ar[r, tail] & \mathcal{H}
	\end{tikzcd}
	\]
\end{pro}
\begin{proof}
	As in the setting $F$-matroids, this statement directly follows from Propositions \ref{proposition: completing square} and \ref{proposition: completing the square} along with the uniqueness of pushout and pullback. 
\end{proof}

By combining the above propositions, we obtain the following. 

\begin{mytheorem}\label{theorem: proto toric}
	The category $\TRS$ is proto-exact with $\mathfrak{M}_t$ and $\mathfrak{E}_t$ as in Definition \ref{definition: TRS M and E}.
\end{mytheorem}

Alongside Proposition \ref{pro: ker coker pa1 trs} and Proposition \ref{prop:protoex_then_protoab}, we have the following corollary. 

\begin{cor}\label{cor: trs proto ab}
	The category $\TRS$ is proto-abelian.
\end{cor} 

Say that a tropical toric reflexive sheaf $\mathcal{E} = (M,\{F^\rho_\bullet\}_\rho)$ is \emph{modular} if the underlying matroid $\underline{M}$ is itself modular, and denote the full subcategory of $\TRS$ consisting of modular tropical toric reflexive sheaves by $\MTRS$. Since restrictions and contractions of modular matroids at flats are again modular, and since $U^0_1$ is modular, the proto-exact and proto-abelian structures on $\TRS$ descends to $\MTRS$. 

\begin{cor}\label{corollary: MTRS proto-abelian}
	The category $\MTRS$ is proto-abelian, with underlying proto-exact structure coming from the restriction of the classes $\MM_t$ and $\EE_t$ to $\MTRS$. 
\end{cor}

Before moving to the question of stability in $\MTRS$, we give a characterization of sums and intersections of strict subobjects in $\TRS$. As expected, the characterization follows immediately from Proposition \ref{prop: sum prod of subobj FSMat} with $F=\mathbb{T}$ since restrictions and contractions of tropical reflexive sheaves are determined by the restrictions and contractions of their underlying $\mathbb{T}$-matroids. 

\begin{pro}\label{prop: sum prod of subobj TRS}
	Let $\mathcal{E}_1,\mathcal{E}_2\rightarrowtail \mathcal{E}$ be strict subobjects of $\mathcal{E} = (M,\{E^\rho_j\}_\rho)$ in $\TRS$  (resp. $\MTRS$) given as restrictions $\mathcal{E}_1= \mathcal{E}|F_1$ and $\mathcal{E}_2= \mathcal{E}|F_2$ at flats $F_1,F_2\in\mathcal{L}(M)$. Then the sum and intersection of the subobjects $\mathcal{E}_1$ and $\mathcal{E}_2$ are given by
	\begin{align*}
		(\mathcal{E}|F_1)+(\mathcal{E}|F_2) &= \mathcal{E}|(F_1 \vee F_2) \\
		(\mathcal{E}|F_1)\cap(\mathcal{E}|F_2) &= \mathcal{E}|(F_1 \wedge F_2).
	\end{align*}
\end{pro}

\subsection{Stability for tropical toric reflexive sheaves}\label{sec:semistab for mtrs}

The proto-abelian structure on the category $\MTRS$ of modular tropical toric reflexive sheaves on $\text{trop}(X_\Sigma)$ can be used to study slope stability phenomena in the context of tropical geometry. Following \cite{khan2024tropical}, we define the slope of a modular tropical toric reflexive sheaf as the ratio of the degree and rank functions (see Subsection \ref{subsec: ttvb}). We then illustrate that the existence and uniqueness of the Harder-Narasimhan filtration of \cite{khan2024tropical} in the category $\MTRS$ follows from the results of \cite{andre2009slope} and \cite{li2023categorification}.

Let $X_\Sigma$ be a smooth toric variety equipped with a constant polarization $h = (c,c,\cdots,c) \in\mathbb{Z}^{\Sigma(1)}$. For a tropical toric reflexive sheaf $\mathcal{E}=(M,\{F^\rho_\bullet\}_\rho)$, recall that its rank $\rk(\mathcal{E})$ is defined to be the rank of the $\mathbb{T}$-matroid $M$ and that the degree of $\mathcal{E}$ is defined to be the following quantity
\begin{equation}
\deg(\mathcal{E}) = \sum_\rho h_\rho\sum_{j\in\mathbb{Z}} j\cdot(\rk_{\underline{M}}(F^\rho_j) - \rk_{\underline{M}}(F^\rho_{j+1})).
\end{equation}

We now define the \emph{slope} of $\mathcal{E}$ to be $\mu(\mathcal{E}) = \rk(\mathcal{E})/\deg(\mathcal{E})$. Say that $\mathcal{E}$ is \emph{stable} (resp.~\emph{semistable}) with respect to $\mu$ if for every nonzero strict subobject $\mathcal{F}$ of $\mathcal{E}$, we have $\mu(\mathcal{F})<\mu(\mathcal{E})$ (resp.~$\mu(\mathcal{F})\leq\mu(\mathcal{E})$). 

Using the above machinery, \cite{khan2024tropical} constructs a Harder-Narasimhan filtration for tropical toric reflexive sheaves in analogy with the usual case of vector bundles on curves. However, in their construction they heavily employ the notion of modularity of various involved flats. They establish the following using methods from tropical geometry and direct investigation of the structure of tropical reflexive sheaves. 

\begin{mythm}[\cite{khan2024tropical}, Theorem 7.25]\label{thm:khmac HN filtr}
	Let $\mathcal{E} = (M, \{F^\rho_\bullet\}_\rho)$ be a tropical toric reflexive sheaf on the tropical toric variety $\text{trop}(X_\Sigma)$. Assume that all $F^\rho_j$ are modular flats of $M$. Then there is an increasing filtration in $\TRS$ as follows:
		\[0 = \mathcal{F}_0\rightarrowtail \mathcal{F}_1 \rightarrowtail \cdots \rightarrowtail \mathcal{F}_k = \mathcal{E}.\]
	The successive quotients $\mathcal{F}_i/\mathcal{F}_{i-1}$ are semistable and satisfy
		\[\mu(\mathcal{F}_1/\mathcal{F}_0) > \mu(\mathcal{F}_2/\mathcal{F}_1) > \cdots > \mu(\mathcal{F}_k/\mathcal{F}_{k-1}).\]
	Under the assumption that each $\mathcal{F}_i$ is of the form $\mathcal{E}|F_i$ for some modular flat $F_i$ of $M$, this filtration is unique.
\end{mythm}

\begin{rmk}\label{rmk: non-uniqueness khmac HN}
Note that in case not all involved flats are modular, the filtration of Theorem \ref{thm:khmac HN filtr} may fail to be unique. This is in contrast to the classical situation, where the Harder-Narasimhan filtration of a vector bundle is unique up to isomorphism. However, in case the underlying matroid $M$ of a tropical toric reflexive sheaf is modular, uniqueness is guaranteed. In particular, the Harder-Narasimhan filtrations of objects in $\MTRS$ are unique.
\end{rmk}

We now illustrate that the slope function $\mu$ on $\MTRS$ defines a categorical slope filtration on $\MTRS$, which immediately yields that the Harder-Narasimhan filtrations of Theorems \ref{thm:li HN filtr} and \ref{thm:khmac HN filtr} coincide.

Observe that by virtue of the additivity of the rank of $\mathbb{T}$-matroids on short exact sequences $M\rightarrowtail N \twoheadrightarrow R$, we immediately see that for every a short exact sequence $\mathcal{E}\rightarrowtail\mathcal{F}\twoheadrightarrow\mathcal{G}$ in $\MTRS$ we have 
	\[\rk(\mathcal{E}) + \rk(\mathcal{G}) = \rk(\mathcal{F}).\]
Additionally, if $\rk(\mathcal{E})=0$ for some tropical reflexive sheaf $\mathcal{E}$ then necessarily the underlying (pointed simple) matroid of $\mathcal{E}$ must have rank zero, so that the underlying matroid is the zero object of $\mathbb{T}$-$\textbf{SMat}_\bullet$. In particular, we must have that $\mathcal{E}$ is the zero object of $\MTRS$. 

The additivity of the degree function on short exact sequences is guaranteed by the following result.

\begin{pro}[\cite{khan2024tropical}, Remark 7.19 and Lemma 7.20]
	Let $\mathcal{E}=(M,\{E^\rho_\bullet\}_\rho)$ be a modular tropical reflexive sheaf and let $F\in\mathcal{L}(M)$ be a flat. Then 
		\[\deg(\mathcal{E}|F) + \deg(\mathcal{E}/F) = \deg(\mathcal{E}).\]
\end{pro} 
\begin{proof}
	This statement is established by the first half of the proof of Lemma 7.20 in \cite{khan2024tropical}. In our context, this proof is guaranteed to hold as written by the modularity assumption on $M$.
\end{proof}

\begin{lem}
	If $\mathcal{E}$ and $\mathcal{F}$ are isomorphic tropical reflexive sheaves then we have $\deg(\mathcal{E}) = \deg(\mathcal{F})$. 
\end{lem} 
\begin{proof}
	By Proposition \ref{pro: trs isos}, we know that $\mathcal{F} = \mathcal{E}\otimes \mathcal{L}(\chi^u)$, where $\mathcal{L}(\chi^u)$ is the tropical line bundle corresponding to the principal divisor of some character $\chi^u$ of $X_\Sigma$. We know that $\mathcal{L}(\chi^u)$ corresponds to a vector $(a_\rho)\in\mathbb{Z}^{\Sigma(1)}$ which by principality of the associated divisor must satisfy $\sum_{\rho\in\Sigma(1)} a_\rho = 0$. We then compute the degree of $\mathcal{F}$ using the relation (\ref{eqn: trs deg tensor}) as in Remark \ref{rmk: deg const on iso classes}:
	\begin{align*}
		\deg(\mathcal{F}) &= \deg(\mathcal{E}\otimes\mathcal{L}(\chi^u)) \\
		&= \deg(\mathcal{E}) + \rk(\underline{M})\cdot\deg(\mathcal{L}(\chi^u)) \\
		&= \deg(\mathcal{E}) + \rk(\underline{M})\cdot\sum_{\rho} h_\rho\cdot a_\rho \\
		&= \deg(\mathcal{E}) + \rk(\underline{M})\cdot c\cdot\sum_{\rho} a_\rho  \\
		&= \deg(\mathcal{E}).
	\end{align*}
	Note that here we have used that the polarization $h$ on $X_\Sigma$ is of the form $h=(c,c,\cdots,c)$.
\end{proof}

From the above two results it immediately follows that $\deg$ is additive on short exact sequences in $\MTRS$. Thus the functions $\deg$ and $\rk$ on $\MTRS$ satisfy the necessary properties of degree and rank functions as in Subsection \ref{subsec: stability and slopes}. In particular, in order to establish that the slope function $\mu$ of \cite{khan2024tropical} is a slope function in the sense of \cite{li2023categorification} we need only to verify that $\mu$ satisfies the strong slope inequality \eqref{eqn:strong slope ineq} in $\MTRS$. This follows by using the characterization of sums and intersections of subobjects as given in Proposition \ref{prop: sum prod of subobj TRS}.

\begin{pro}\label{pro: stron slope ineq trs}
	The slope function $\mu$ satisfies the strong slope inequality on $\MTRS$. In particular, given a modular tropical reflexive sheaf $\mathcal{E}$ with strict subobjects $\mathcal{F},\mathcal{G}$ given as restrictions $\mathcal{F} = \mathcal{E}|F$ and $\mathcal{G} = \mathcal{E}|G$ at flats $F,G\in\mathcal{L}(M)$, we have
	\[\mu\left(\frac{\mathcal{F}}{\mathcal{F}\cap \mathcal{G}}\right) \leq \mu\left(\frac{\mathcal{F} + \mathcal{G}}{\mathcal{G}}\right).\]
\end{pro}
\begin{proof}
	We have
		\begin{align*}
			\frac{\mathcal{F}}{\mathcal{F}\cap \mathcal{G}} &= \frac{\mathcal{E}|F}{\mathcal{E}|(F\wedge G)} = (\mathcal{E}|F)/(F\wedge G) \\
			\frac{(\mathcal{F} + \mathcal{G})}{\mathcal{G}} &= \frac{\mathcal{E}|(F\vee G)}{\mathcal{E}|F} = (\mathcal{E}|(F\vee G))/F = (\mathcal{E}/F)|(F\vee G)
		\end{align*} 
	The desired inequality then follows from \cite[Lemma 7.24]{khan2024tropical}.
\end{proof}

This establishes that $\mu$ is a categorical slope function on $\MTRS$ in the sense of \cite{li2023categorification}. In particular, since $\TRS$ and $\MTRS$ inherit the Artinian and Noetherian conditions from $\mathbb{T}$-$\textbf{SMat}_\bullet$ guaranteed by Remark \ref{rmk: FSMat art noeth}, Theorem \ref{thm:li HN filtr} guarantees that to $\mu$ there exists a categorical slope filtration on $\MTRS$ which satisfies all properties of the slope filtration given in Theorem \ref{thm:khmac HN filtr}. The uniqueness of such a filtration as guaranteed by both theorems then yields the following. 

\begin{cor}\label{cor: HN filtr coincide}
	The Harder-Narasimhan filtrations of Theorem \ref{thm:khmac HN filtr} associated to modular tropical reflexive sheaves are given by the categorical slope filtration on $\MTRS$ given by Theorem \ref{thm:li HN filtr}. Stated differently, the Harder-Narasimhan filtration of \cite{khan2024tropical} for modular tropical reflexive sheaves is a categorical slope filtration in the sense of \cite{li2023categorification}.
\end{cor}

This illustrates that the proto-abelian structure of $\mathbb{T}$-matroids ultimately yields an analogue of the classical Harder-Narasimhan filtration of vector bundles in the context of tropical geometry. As noted in Section 8 of \cite{khan2024tropical}, this machinery can also be used to give Harder-Narasimhan filtrations of (modular) tropical vector bundles on those tropical varieties which can be embedded into tropical toric varieties $\text{trop}(X_\Sigma)$. 
 
\bibliography{quiver}\bibliographystyle{alpha}

\end{document}